\documentclass[12pt]{amsart}
\usepackage{amsthm}
\usepackage{amsmath}
\usepackage{amssymb}
\usepackage{comment}
\usepackage{enumerate}
\usepackage{amsfonts}
\usepackage{dsfont}
\usepackage{mathtools}
\usepackage{afterpage}
\usepackage{hyperref}
\usepackage[capitalize]{cleveref}
\usepackage{todonotes}
\usepackage{algorithm,algorithmic}
\usepackage{multicol}
\usepackage{circuitikz}
\usepackage{mathrsfs}
\usepackage{tikz-cd}
\usetikzlibrary{math}
\usepackage{bm}
\usetikzlibrary{shapes.geometric}
\usepackage[utf8]{inputenc}
\usepackage[margin=1.1in]{geometry}
\DeclareTextFontCommand{\emph}{\color{blue}\em}

\newtheorem{theorem}{Theorem}[section]

\newtheorem{def-prop}[theorem]{Definition-Proposition}
\newtheorem{prop}[theorem]{Proposition}
\newtheorem{proposition}[theorem]{Proposition}

\newtheorem{lemma}[theorem]{Lemma}
\newtheorem{cor}[theorem]{Corollary}

\theoremstyle{definition}
\newtheorem{ex}[theorem]{Example}
\newtheorem{example}[theorem]{Example}

\newtheorem{notation}[theorem]{Notation}
\newtheorem{definition}[theorem]{Definition}
\newtheorem{defn}[theorem]{Definition}
\newtheorem{remark}[theorem]{Remark}

\theoremstyle{claim}

\def\F{\mathcal{F}}
\def\0{{\mathbf{0}}}
\renewcommand{\P}{\mathbf{P}}
\newcommand{\Pbb}{\mathbb{P}}
\newcommand{\Nbb}{\mathbb{N}}

\newcommand{\N}{\mathcal{N}}

\newcommand{\e}{\mathbf{e}}
\newcommand{\p}{\mathbf{p}}
\newcommand{\q}{\mathbf{q}}

\newcommand{\x}{\mathbf{x}}
\newcommand{\y}{\mathbf{y}}
\newcommand{\w}{\mathbf{w}}
\renewcommand{\u}{\mathbf{u}}

\newcommand{\C}{\mathbb{C}}

\newcommand{\z}{\mathbf{z}}
\newcommand{\R}{\mathbb{R}}
\newcommand{\Q}{\mathbb{Q}}

\newcommand{\VolC}{\mathrm{Vol}^{\vee}}

\newcommand\ip[1]{\langle #1 \rangle}
\newcommand{\base}{y_0}

\DeclareMathOperator{\Vol}{Vol}

\DeclareMathOperator{\Int}{Int}

\def\Z{{\mathbb{Z}}}
\def\F{{\mathcal{F}}}
\def\M{{\mathcal{M}}}

\def\T{{\mathcal{T}}}
\def\R{{\mathbb R}}

\def\Res{{\rm Res}}
\def\sp{{\rm span}}

\def\b{{\mathbf{b}}}
\def\a{{\mathbf{a}}}
\def\h{{\mathbf{h}}}
\def\n{{\mathbf{n}}}
\def\p{{\mathbf{p}}}
\def\q{{\mathbf{q}}}
\def\z{{\mathbf{z}}}
\def\t{{\mathbf{t}}}

\def\v{{\mathbf{v}}}
\def\w{{\mathbf{w}}}

\def\bby{{\bar \y}}

\def\tp{{\tilde \p}}
\def\ty{{\tilde \y}}

\DeclareMathOperator{\Cone}{cone}
\DeclareMathOperator{\codim}{codim}
\DeclareMathOperator{\conv}{conv}

\newcommand{\keydef}{\text{dual lattice function }}

\newcommand{\TL}[1]{{\bf TL: #1}}
\newcommand{\lei}[1]{{\color{blue} \sf Lei: [#1]}}

\begin{document}
\title{Dual lattice functions of polytopes}
\author{Yibo Gao}
\address{Beijing International Center for Mathematical Research, Peking University, \mbox{Beijing, 100871}}
\email{\href{mailto:gaoyibo@bicmr.pku.edu.cn}{{\tt gaoyibo@bicmr.pku.edu.cn}}}
\author{Thomas Lam}
\address{Department of Mathematics, University of Michigan, \mbox{Ann Arbor, MI 48109}}
\email{\href{mailto:tfylam@umich.edu}{{\tt tfylam@umich.edu}}}
\author{Lei Xue}
\address{Department of Mathematics, Colby College, \mbox{Waterville, ME 04901}}
\email{\href{mailto:leixue@colby.edu}{{\tt leixue@colby.edu}}}

\thanks{T.L.\ is supported by Grant No.~DMS-1953852 and DMS-2348799 from the National Science Foundation. Y.G. is partially supported by NSFC Grant no. 12471309.} 
\date{\today}

\begin{abstract}
We define the dual lattice function of a rational polytope $P$ via the discrete Laplace transform of the exponential of its support function. This definition is a discrete analogue of the dual volume function of a polytope that the authors studied in previous work. We show that the dual lattice function is valuative, and by multiplying with the torus form, it becomes the canonical form of the exponential polytope $\mathrm{exp}(P)$ as a positive geometry.  This result suggests the study of the class of toric polytopes, which are certain semialgebraic subsets of projective toric varieties. Our work is a first step towards discretization of positive geometries in the simplest case of polytopes.
\end{abstract}
\maketitle
\setcounter{tocdepth}{1}

\section{Introduction}

\subsection{Volumes and lattice points}
Among the valuative invariants of polytopes, the two most classical ones are the \emph{volume}
and the number (or generating function) of \emph{integer points}. There is a long history of interaction between the study of volumes and of lattice points, for example, Ehrhart polynomials, Brion's Theorem, Euler-Maclaurin formulae,  and Todd operators (see, for example, \cite{Brion88,BV,GP}), and so on.  For a brief survey of this rich field, we refer the reader to \cite{Barvinok} and \cite{BR}.

Recent developments in \emph{positive geometry} \cite{ABL,LamOPAC} have suggested the study of the dual volume function, which can be intuitively understood as the volume of the polar polytope as a function of the location of the origin. In our recent work \cite{GLX}, we proposed the study of the more general \emph{dual mixed volume function} of a tuple $(P_1,P_2,\ldots,P_r)$ of polytopes. In this work, we study lattice point generating functions following a similar dual philosophy.

Let $P \subset \R^d$ be a polytope, which we assume to be full-dimensional in the introduction.  We define the \emph{\keydef} $K_P(\z)$ of $P$ as the sum
$$
K_P(\z):=\sum_{\v\in \Z^d}\exp(-h_{P-\z}(\v)),
$$
where $h_P$ denotes the support function of $P$.  We show in \Cref{thm: sum-of-cones} that $K_P(\z)$ is a rational function in the variables $\y = \exp(\z)$, and in \Cref{thm:limit-to-VolC} we prove that the dual volume rational function of \cite{GLX} can be obtained from $K_P(\z)$ as a limit.  Whereas dual volume is the Laplace transform of the support function, our \keydef $K_P(\y)$ can be viewed as a discrete Laplace transform; see Remark~\ref{rmk:Laplace}. Moreover, there is a natural connection between the \keydef and Ehrhart theory (\Cref{prop:Ehrhart}).

In \cref{sec:properties}, we prove that the assignment $P \mapsto K_P(\z)$ is a valuative function on full-dimensional polytopes, along with other foundational properties of $K_P$. Analogously to the dual mixed volume of \cite{GLX}, we may also define the \emph{dual mixed lattice function}
$$
K_{P_1,\ldots,P_r}(\x) = \sum_{\v\in \Z^d}\exp(-h_{x_1P_1+\cdots+x_rP_r}(\v)).
$$
A systematic study of $K_{P_1,\ldots,P_r}(\x)$ is left for future work.

\subsection{Positive geometries and canonical forms}
The data of a \emph{positive geometry} consists of a triple $(X,X_{\geq0},\Omega_X)$, where $X$ is a projective variety, $X_{\geq0}$ is a closed semi-algebraic subset, and $\Omega_X$ is a top-degree meromorphic differential form on $X$ uniquely determined by a recursion on its residues. Readers are referred to \cite{ABL,LamOPAC} for details, and to \cite{ranestad2025,lamICM} for a broad survey. Recently, many natural spaces have been shown to admit the structure of a positive geometry (see for instance \cite{maazouz2024positive,telen2025,shen2026abct,sturmfels2026cubic}). The present work provides a new instance.  As one of the most classical examples, every polytope $P \subset \R^d \subset \Pbb^d$ is a positive geometry with canonical form given by
\[\Omega_P = \Vol^\vee_\z(P) dz_1 \cdots dz_d,\]
where $\Vol^\vee_\z(P)$ denotes the dual volume of the translation $P- \z$ \cite{GLX}. 

Remarkably, the \keydef $K_P(\y)$ plays a similar role for the exponential polytope $\exp(P)$.  Coordinatewise exponentiation defines an injective map $\R^d \to \R_{>0}^d$.  The \emph{exponential polytope} $\exp(P)$ is the image of $P$ under this map.  We show (\cref{thm:expP2}) that the exponential polytope is a positive geometry with canonical form 
\begin{equation}\label{eq:Omegaintro}
\Omega_{\exp(P)} = K_P(\y) \frac{dy_1 \cdots dy_d}{y_1\cdots y_d}.
\end{equation}
Here, $P$ is assumed to be a polytope with vertices belonging to a lattice $L \cong \Z^d$, and $\R_{>0}^d$ is viewed as the positive real points of a complex torus $T_L \cong (\C^\times)^d$ with compactification a projective toric variety $X(L)$.  The positive geometry $(X(L), \exp(P))$ has the property that every face is a positive geometry sitting inside a toric variety. Thus exponential polytopes are examples of \emph{toric polytopes} (\Cref{def:toricpolytope}) that we introduce in \cref{sub:toric-polytopes}.

Another example of a toric polytope is the positive part $X(L)_{\geq 0}$ of the toric variety $X(L)$.  This is the closure of the positive part $(T_L)_{>0} \cong \R_{>0}^d$ of the dense torus $T_L \subset X(L)$.  More generally, while we do not prove it here, we expect that the \keydef can be generalized to any polyhedron $P \subset \R^d$, and the closure of the exponential polyhedron in $X(L)(\R)$ would be a positive geometry with canonical form \eqref{eq:Omegaintro}.

\subsection{Motivation from physics}

Canonical forms of positive geometries were defined to capture combinatorial and geometric features of scattering amplitudes.  For example, the canonical form of the associahedron polytope encodes scattering amplitudes in planar $\phi^3$-theory; the conjectural canonical form of the amplituhedron encodes super Yang-Mills scattering amplitudes.

Whereas planar $\phi^3$-amplitudes are related to dual volumes of associahedra, the dual lattice function of an associahedron has the physical interpretation as an \emph{$\alpha'$-corrected amplitude}, and is inverse to the string theory KLT (Kawai-Lewellen-Tye) matrix.  See \cite{lam2025matroidsamplitudes} for further details, where these assertions are explained in terms of the Bergman fan of a matroid.

Our work can be viewed as a first step towards a discrete analogue of positive geometry.  A natural next step would be to investigate dual lattice functions in the Grassmannian.

\subsection{Organization of the paper}
In \Cref{sec:prelim}, we introduce the necessary preliminary material. In \Cref{sec:definitions}, we define the function $K_P$, the main object of study, and provide formulae for its computation; we also explain connections to Ehrhart theory and dual mixed volume. In \Cref{sec:properties}, we study further properties (valuation, degree, residues) of $K_P$. In \Cref{sec:toric}, we introduce $\exp(P)$, the exponential polytope of $P$, and prove the main theorem (\Cref{thm:expP2}) that they are positive geometries with canonical form $\Theta_P:=K_P(\y)\omega_{T_P}$, and are also examples of the toric polytopes that we define.

\section{Preliminaries}\label{sec:prelim}
\subsection{Lattices and tori}
Let $L$ be a lattice of rank $d$, and let $L^{\vee}:=\mathrm{Hom}_{\mathbb{Z}}(L,\Z)$ be its dual lattice. For an affine subspace $H\subset L\otimes_{\Z}\R$ defined over $\Q$, and any $\z_0\in H$ denote
\begin{align*}
H_0&=H-\z_0:=\{\z-\z_0\:|\:\z\in H\}\\
H_0^{\perp}&:=\{\v\in L^{\vee}\otimes_{\Z}\R\:|\: \v(\z)=0\ \forall \z\in H_0\}.
\end{align*}
Note that the linear spaces $H_0, H_0^\perp$ do not depend on the choice $\z_0$. 

We have a sublattice $H_0\cap L$ of $L$. Correspondingly, we can translate this lattice by $\z_0$ to move it inside $H$, and denote it by $L_{H,\z_0}:=\{\z\in H\:|\: \z-\z_0\in L\}$, the \emph{sublattice of $L$ on $H$ with origin $\z_0$}. The group structure $(+_{L'})$ on $L'=L_{H,\z_0}$ is given by \[\x+_{L'}\y=\x+\y-\z_0.\] Correspondingly, the vector space structure $(+_{H},\cdot_{H})$ on $H$ is given by \[\x+_H \y=\x+\y-\z_0,\quad \alpha\cdot_H \x=\alpha\x+(1-\alpha)\z_0.\]
\begin{lemma}\label{lem:dual-lattices}
With notations as above, $H_0\cap L$ and $L^{\vee}/H_0^{\perp}$ are dual lattices. 
\end{lemma}
\begin{proof}
It suffices to show that the map $f:H_0\cap L\rightarrow \mathrm{Hom}(L^{\vee}/H_0^{\perp},\Z)$ given by $\z\mapsto (\v\mapsto\v(\z))$ is a bijection. First, this map is well-defined, since if $\v\in H_0^{\perp}$ and $\z\in H_0$, $\v(\z)=0$. This map is also linear. Pick a nonzero $\z\in H_0\cap L$ and pick $\v\in L^{\vee}$ such that $\v(\z)\neq0$. Write the image of $\v$ in $L^{\vee}/H_0^{\perp}$ also as $\v$. Then $f(\z)(\v)=\v(\z)\neq0$, $f(\z)\neq0$ so $f$ is injective. To show that $f$ is surjective, pick any $\sigma:L^{\vee}/H_0^{\perp}\rightarrow\Z$. We have $\bar\sigma:L^{\vee}\rightarrow L^{\vee}/H_0^{\perp}\rightarrow\Z$. By the natural isomorphism between $L^{\vee\vee}$ and $L$, there exists some $\z_\sigma\in L$ such that $\bar\sigma$ is given by $\v\mapsto\v(\z_\sigma)$. In order for $\bar\sigma$ to descend to a map $\sigma:L^{\vee}/H_0^{\perp}\rightarrow\Z$, we must have that $\v(\z_\sigma)=0$ for all $\v\in H_0^{\perp}$. Thus, $\z_{\sigma}\in H_0$ as $H_0^{\perp\perp}=H_0$. It is now straightforward to see that $f(\z_\sigma)=\sigma$ so we are done.
\end{proof}

\subsection{Polytopes and their normal fans}
Let $L$ be a lattice of rank $d$. Our ambient vector space is $L\otimes_{\Z}\R\simeq\R^d$. For a non-empty closed bounded convex set $S\subset L\otimes_{\Z}\R$, its \emph{support function} $h_S$ on $L^{\vee}\otimes_{\Z}\R=\mathrm{Hom}_{\Z}(L,\R)$ is given by $\v\mapsto -\min_{\x\in S}\v(\x)$. 

For a face $F$ of a polytope $P\subset L\otimes_\Z\R$, its \emph{dual cone} $C_F\subset L^{\vee}\otimes_\Z\R$ consists of points whose \emph{support face} in $P$ contains $F$. In other words, \[C_F:=\{\v\in L^{\vee}\otimes_\Z\R\:|\: h_P(\v)=-\v(\z)\text{ for all }\z\in F\}.\]
The \emph{normal fan} $\N(P)$ consists of all such cones $C_F$ for faces $F$ of $P$. The maximal cones in $\N(P)$ are the dual cones corresponding to vertices of $P$. Let $H$ be the affine span of $P$ with a chosen $\z_0$ as the origin. The \emph{reduced dual cone} for a face $F$ of $P$ is \[\widetilde{C}_F:=\{\v\in L^{\vee}/H_0^{\perp}\otimes_\Z\R\:|\: \min_{\x\in P}\v(\x)=\v(\z)\text{ for all }\z\in F\}\]
and the reduced normal fan $\widetilde{\mathcal{N}}(P)$ consists of all such cones $\widetilde{C}_F$ for faces $F$ of $P$. If $P$ is full-dimensional in $L\otimes_\Z\R$, the reduced normal fan is just the normal fan.

\section{Dual lattice point function}\label{sec:definitions}
\subsection{The main definition}
In this section, we define our main object of study.
\begin{definition}\label{def:dual-lattice-point-generating-function}
Let $P\subset L\otimes_{\Z}\R$ be a rational polytope and let $H$ be its affine span. Pick $\z_0\in H$ and let $L'=L_{H,\z_0}$ be the sublattice on $H$ with origin $\z_0$. The \emph{\keydef of $P$} is the function on $H$ defined by 
\begin{equation}\label{eqn:main-def}
K_P(\base;\z):=\sum_{\v\in (L')^{\vee}}\exp_{\base}(-h_{P-(\z-\z_0)}(\v))=\sum_{\v\in (L')^{\vee}}\exp_{\base}(-h_{P}(\v)-\v(\z{-}\z_0)).
\end{equation}
Here we use the nonstandard notation of $\exp_t(A):=t^A=\exp(\ln(t)A)$ to emphasize the exponent. Throughout, these notations will be used interchangeably depending on the context. We write $K_P(\z):=K_P(e;\z)$ for simplicity. 
\end{definition}


By \Cref{lem:dual-lattices}, $(L')^{\vee}$ can be identified with $L^{\vee}/H_0^{\perp}$, where the pairing between $L'$ and $L^{\vee}/H_0^{\perp}$ is given by $(\z,\v)\mapsto \v(\z-\z_0)$. Thus, \Cref{eqn:main-def} can also be written as
\begin{equation}
K_P(\base;\z)=\sum_{\v\in L^{\vee}/H_0^{\perp}}\exp_{\base}\big(\min_{\x\in P}\v((\x{-}\z_0){-}(\z{-}\z_0))\big)=\sum_{\v\in L^{\vee}/H_0^{\perp}}\exp_{\base}(\min_{\x\in P}\v(\x{-}\z)),
\end{equation}
and is independent of the choice of the origin $\z_0$ in the sublattice on $H$. 

\begin{remark}
An important property of Definition \ref{def:dual-lattice-point-generating-function} is that $K_P$ depends only on the full-dimensional polytope $P$ in the space $H=L'\otimes_\Z\R$, as the definition does not depend on the bigger lattice $L\supset L'$. 
\end{remark}

\subsection{Formula for calculating $K_P$}\label{subsect: formular for K_P}
For a rational simplicial cone (not necessarily full-dimensional) $C\subset L^{\vee}\otimes_\Z\R$ with integral generators $\u_1,\ldots,\u_m\in L^{\vee}$, the \emph{fundamental parallelepiped} $\Pi(C)$ (resp. the \emph{open-closed fundamental parallelepiped} $\Pi^\circ(C)$) is given by \[\Pi(C):=\left\{\sum_{i=1}^m \alpha_i\u_i\:|\:\alpha_i\in[0,1)\right\}, \qquad \text{and} \qquad \Pi^\circ(C):=\left\{\sum_{i=1}^m \alpha_i\u_i\:|\:\alpha_i\in(0,1]\right\}.\]
Here, the integral generators $\u_1,\ldots,\u_m$ are determined by the following properties: $C=\{\sum_{i=1}^m\alpha_i\u_i\:|\:\alpha_i\in\R_{\geq0}\}$, $\dim C=m$, $\u_i\in L^{\vee}$ and $\u_i/k\notin L^{\vee}$ for any $k\in\Z_{\geq2}$, for all $i$. 

Let $C \subset  L^{\vee}\otimes_\Z\R$ be a rational simplicial cone. To compute $K_P(\z)$, we use the following lattice point generating function, which is called the \emph{integer-point transform} of $C$ in \cite{BS}.
\begin{align}\label{eq: Def of A_c}
A_C(\t):=\sum_{\v\in C\cap L^{\vee}}\t^{\v}.
\end{align}

The following is a well-known result about the integer-point transform.
\begin{prop}\label{prop:lattice-point-generating-function}
Let $C\subset L^{\vee}\otimes_\Z\R$ be a rational simplicial cone with integral generators $\u_1,\ldots,\u_m$, and denote by $\Pi= \Pi(C)$ the fundamental parallelepiped for $C$. Then 
\[
A_C(\t)=\left(\sum_{\v\in\Pi\cap L^{\vee}}\t^{\v}\right)\prod_{i=1}^m\frac{1}{1-\t^{\u_i}}.
\]
\end{prop}
\begin{proof} 
Translations of the fundamental parallelepiped $\Pi$ tile the cone $C$.  Thus every lattice point $\v \in C \cap L^{\vee}$ lies in exactly one translate $\q + \Pi$ where $\q = \sum q_i\u_i$ with $q_i\in \Z_{\geq 0}$. In other words, every $\v$ can be uniquely decomposed as $\v = \q + \p$ where $\p\in \Pi \cap L^{\vee}$. The statement follows. 
\end{proof}

The following is a straightforward consequence of \Cref{prop:lattice-point-generating-function}, by repeated application of the fact that $A_C(\t)=0$ if $C$ contains a line. This classical reciprocity is first shown by Stanley in \cite{Stanley75}.

\begin{cor}[Stanley's Reciprocity]\label{reciprocity of A_C} The integer-point transform of the interior of $C$ is 
\[A_{\Int(C)}(\t) :=\sum_{\v\in \Int(C) \cap L^{\vee}}\t^{\v}= (-1)^m A_{C}(\t^{-1}).\]
\end{cor}

The function $A_C$ has been very well-studied since Stenley's Reciprocity Theorem, appearing in classicial works like Brion's theorem \cite{Brion88} and Barvinok's polynomial-time counting algorithm \cite{Barvinok94}. For a comprehensive overview, see \cite{BR}, and \cite{BS}.

\begin{lemma}\label{lem: parallelepiped formula}
Let $P\subset L\otimes_\Z\R$ be a rational polytope of full dimension, $F$ be a face and $C\subset C_F$ be a rational simplicial cone with integral generators $\u_1,\ldots,\u_m$. Then 
\[K_{P,C}(\base;\z):=\sum_{\v\in C\cap L^{\vee}}\exp_{\base}(\min_{\x\in P}\v(\x-\z))=\left(\sum_{\v\in\Pi(C)\cap L^{\vee}}\base^{\v(\p{-}\z)}\right)\prod_{i=1}^m\frac{1}{1{-}\base^{\u_i(\p{-}\z)}}\]
for any point $\p\in F$. 
\end{lemma}
\begin{proof}
By definition of $C_F$, for any $\v \in C$, $\min_{\x\in P}\v(\x)=\v(\p)$ for any point $\p\in F$. The rest follows from Proposition \ref{prop:lattice-point-generating-function}. 
\end{proof}

\begin{lemma}\label{lem: parallelepiped formula interior}
Let $P\subset L\otimes_\Z\R$ be a rational polytope of full dimension, $F$ be a face and $C\subset C_F$ be a rational simplicial cone with integral generators $\u_1,\ldots,\u_m$. Then 
\[K^\circ_{P,C}(\base;\z):=\sum_{\v\in\Int(C)\cap L^{\vee}}\exp_{\base}(\min_{\x\in P}\v(\x-\z))=\left(\sum_{\v\in\Pi^\circ(C)\cap L^{\vee}}\base^{\v(\p{-}\z)}\right)\prod_{i=1}^m\frac{1}{1{-}\base^{\u_i(\p{-}\z)}}\]
for any point $\p\in F$. 
\end{lemma}

We compute $K_P(\z)$ by triangulation in the dual space where the normal fan $\N(P)$ lives.
\begin{theorem}\label{thm: sum-of-cones}
Let $P\subset L\otimes_\Z\R$ be a rational polytope of full dimension.  Let $\T$ be a complete, rational simplicial fan that refines $\N(P)$.
Then \[K_P(\base;\z)=\sum_{C\in \T}(-1)^{\mathrm{codim}(C)}K_{P,C}(\base;\z)
= \sum_{C\in \T} K^\circ_{P,C}(\base;\z).\]
In particular, $K_P(\base;\y)$ is a rational function in $y_0$ and in the variables $$\y=\exp_{\base}(\z) = (y_1 = \base^{z_1},\ldots, y_d = \base^{z_d}).$$
\end{theorem}
\begin{proof}
The equality $K_P(\base;\z)=\sum_{C\in \T} K^\circ_{P,C}(\base;\z)$ is evident, and \[K_P(\base;\z)=\sum_{C\in \T}(-1)^{\mathrm{codim}(C)}K_{P,C}(\base;\z)\] follows from a straightforward inclusion-exclusion argument. Each summand $K_{P,C}(\base;\z)$ is a rational function in the variables $\y=\exp(\z)$, and so is the overall sum $K_P(\base;\z)$. 
\end{proof}
\begin{notation}
In this paper, we always use $\z=(z_1,\ldots,z_d)$, and $\y=(y_1,\ldots,y_d)=\exp_{\base}(\z)$, as the input variables for the generating functions $K_P$. Note that, after the change of variables, the formula in Definition \ref{def:dual-lattice-point-generating-function} becomes
\begin{align}\label{eqn: K_P(y)}
    K_P (y_0; \y) = \sum_{\v\in (L')^\vee}y_0^{-h_{P-\z_0}(\v)} \y^{-\v}.
\end{align}
\end{notation} 

\begin{definition}\label{def:KPF}
For a face $F \subset P$, define
\[
K_{P,F}(\base;\z):= \sum_{\v\in C_F \cap L^{\vee}}\exp_{\base}(\min_{\x\in P}\v(\x-\z)), \quad \text{and} \quad K^\circ_{P,F}(\base;\z):= \sum_{\v\in \Int(C_F) \cap L^{\vee}}\exp_{\base}(\min_{\x\in P}\v(\x-\z)).\]
\end{definition}
We then have the alternative formula:
\begin{theorem}\label{thm: sum-of-cones-faces}
Let $P\subset L\otimes_\Z\R$ be a rational polytope of full dimension. 
Then \[K_P(y_0;\z)=\sum_{F}(-1)^{\dim(F)}K_{P,F}(y_0;\z) = \sum_F K^\circ_{P,F}(y_0;\z),
\]
where the summation is over all the faces $F$ of $P$.
\end{theorem}

\begin{ex}\label{ex:dim-1}
Let $L$ be the standard lattice $\mathbb{Z}$ in $\R$. Let $P=[a,b]\subset\R$ be an interval. Identify $L^{\vee}$ with $\Z$ as well. Pick $\u_1=(-1)$ and $\u_2=(1)$ as the integral generator vectors of its normal fan. Let $C_1=\R_{\leq0}$ and $C_2=\R_{\geq0}$. Then Lemma~\ref{lem: parallelepiped formula} gives
\[K_{P,C_1}(z)=\frac{1}{1-\exp(z-b)},\quad K_{P,C_2}(z)=\frac{1}{1-\exp(a-z)},\quad K_{P,C_1\cap C_2}(z)=1.\]
We then have, with $y = y_1$, by Theorem~\ref{thm: sum-of-cones}, \[K_P(z)=K_{P,C_1}(z)+K_{P,C_2}(z)-1=\frac{1-\exp(a-b)}{(1-\exp(z-b))(1-\exp(a-z))}=\frac{1-e^{a-b}}{(1-e^ay^{-1})(1-e^{-b}y)}.\]
\end{ex}

\begin{ex}\label{ex: 2-dim}
Let $L$ be the standard lattice $\mathbb{Z}^2$ in $\R^2$. Let $P$ be the convex hull of $(1,0),(0,1),(-2,-1)$ and $(0,-1)$. We pick the vectors $\u_1 = (1,-1)$, $\u_2 = (-1,-1)$, $\u_3 = (-1,1)$, and $\u_4 = (0,1)$ as the integral generator vectors of its normal fan $\mathcal{N}(P)$. See Figure~\ref{fig:dual-generating-function-example}.


We now compute $K_{P,C_1} (\z)$ using the formula from Lemma~\ref{lem: parallelepiped formula}. The only lattice points in $\Pi(C_1)$ are $(0,0)$ and $ (0,-1)$. Each corresponds to a summand in the numerator in $ K_{P,C_1} (\z)$:
\begin{align*}
 K_{P,C_1} (\z) = \frac{1 + e^{-1+z_2} }{ (1-e^{-1-z_1+z_2}) (1-e^{-1+z_1+z_2}) }.
\end{align*}

Similarly we compute the terms corresponding to $C_2, C_3, C_4$ as follows:
\[ K_{P,C_2} (\z) = \frac{1 }{ (1-e^{-1-z_1+z_2}) (1-e^{-1-z_2}) };  \quad  K_{P,C_3}(\z) = \frac{1 }{ (1-e^{-1-z_2}) (1-e^{-1+z_1-z_2}) };  \]
\[ K_{P,C_4}(\z) = \frac{1 + e^{-1+z_1}}{ (1-e^{-1+z_1-z_2}) (1-e^{-1+z_1+z_2}) }.  \]
Summing up, we obtain
\begin{eqnarray*}
K_P(\z)
&=& K_{P,C_1} + K_{P,C_2} + K_{P,C_3}  + K_{P,C_4} \\
& &  \quad - K_{P,C_1\cap C_2} - K_{P,C_2\cap C_3} - K_{P,C_3\cap C_4} - K_{P,C_1\cap C_4} +  K_{P,C_1\cap C_2 \cap C_3\cap C_4}\\
&=&  K_{P,C_1} + K_{P,C_2} + K_{P,C_3} + K_{P,C_4} \\
& & \quad - \frac{1}{1-e^{-1-z_1+z_2}} - \frac{1}{1-e^{-1-z_2}} -\frac{1}{1-e^{-1+z_1- z_2}} - \frac{1}{1-e^{-1+z_1+z_2}} +1\\
 &=& \frac{(1 - e^{-1}) \Big( 1 + e^{-1}(1 + e^{z_1} + e^{z_2}) - e^{-2}(1 + e^{z_1} + e^{z_1-z_2}) - e^{-3+z_1} \Big)}{(1-e^{-1-z_1+z_2})(1-e^{-1+z_1+z_2})(1-e^{-1-z_2})(1-e^{-1+z_1-z_2})}.
\end{eqnarray*}

\begin{figure}[h!]
\centering
\begin{tikzpicture}[scale=1.0]
\draw[fill=black,opacity=0.2](0,1)--(1,0)--(0,-1)--(-2,-1)--(0,1);
\draw[thick](0,1)--(1,0)--(0,-1)--(-2,-1)--(0,1);
\node at (0,1) {$\bullet$};
\node at (1,0) {$\bullet$};
\node at (-2,-1) {$\bullet$};
\node at (0,-1) {$\bullet$};
\node at (-2,-1) {$\bullet$};
\node[above] at (0,1) {$(0,1)$};
\node[above] at (1,0) {$(1,0)$};
\node[below] at (-2,-1) {$(-2,-1)$};
\node[below] at (0,-1) {$(0,-1)$};
\draw[thin](0,-1)--(0,2);
\draw[thin](-2,0)--(2,0);
\end{tikzpicture}
\qquad
\begin{tikzpicture}[scale=1.0]
\draw[fill=black,opacity=0.2](1,-1)--(0,1)--(-1,1)--(-1,-1)--(1,-1);
\node at (0,0) {$\bullet$};


\draw[->,thin](0,0)--(-1,1);
\draw[->,thin](0,0)--(0,1);
\draw[->,thin](0,0)--(1,-1);
\draw[->,thin](0,0)--(-1,-1);

\node at (1,1) {$C_2$};
\node at (0,-1) {$C_1$};
\node at (-0.5,1.5) {$C_3$};
\node at (-1.5,0) {$C_4$};

\node[right] at (1,-1) {$\mathbf{u_1}$};
\node[below] at (-1,-1) {$\mathbf{u_2}$};
\node[left] at (-1,1) {$\mathbf{u_3}$};
\node[above] at (0,1) {$\mathbf{u_4}$};
\end{tikzpicture}
\qquad
\begin{tikzpicture}[scale=1.0]
\draw[fill=black,opacity=0.2](1,-1)--(1,0)--(0,1)--(-1,2)--(-1,1)--(-2,0)--(0,-2)--(1,-1);
\node at (0,0) {$\bullet$};

\node at (-1,0) {$\bullet$};
\node at (0,-1) {$\bullet$};

\draw[->,thin](0,0)--(-1,1);

\draw[->,thin](0,0)--(0,1);
\draw[->,thin](0,0)--(1,-1);
\draw[->,thin](0,0)--(-1,-1);

\node at (1.5,0) {$\Pi(C_2)$};
\node at (1,-2) {$\Pi(C_1)$};
\node at (-0.4,2) {$\Pi(C_3)$};
\node at (-2.5,0) {$\Pi(C_4)$};

\node[right] at (1,-1) {$\mathbf{u_1}$};
\node[below] at (-1,-1) {$\mathbf{u_2}$};
\node[left] at (-1,1) {$\mathbf{u_3}$};
\node[above] at (0,1) {$\mathbf{u_4}$};
\end{tikzpicture}
\caption{A polytope $P$, its normal $\N(P)$, and the fundamental parallelepipeds for each maximal cone in $\N(P)$.}
\label{fig:dual-generating-function-example}
\end{figure}
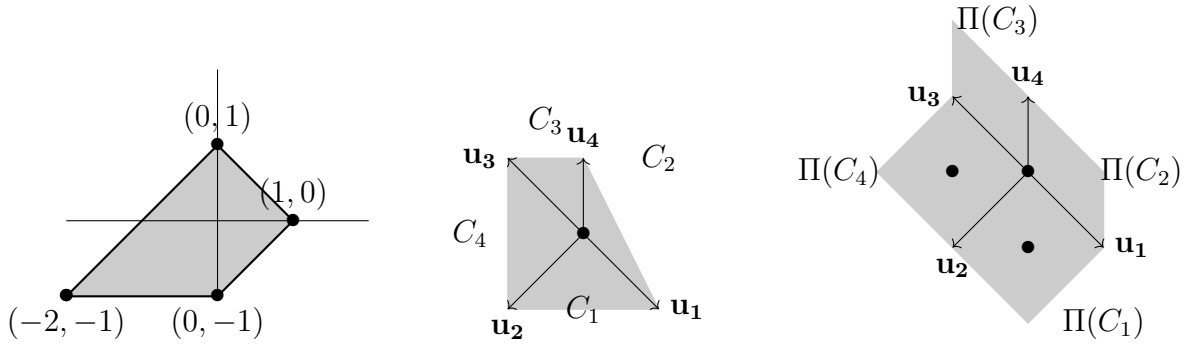
\end{ex}

\begin{ex}\label{ex:3-dim}
Consider a cube $P = [-1,1]^3$ in $\R^3$. Its normal fan contains $8$ maximal cones, each being simplicial of the form $C_i = \Cone\{\u_1, \u_2, \u_3\}$, where $\u_j \in \{\e_j, -\e_j \}$ for $j = 1,2,3$. Every generating ray $\u_j$ corresponds to a factor
\[ (1- e^{-1\pm z_j}) \]
in the denominator. For example, the maximal cone that corresponds to the vertex $\v = (-1,1, -1)$ is $C_{\v} = \Cone (\e_1, -\e_2, \e_3)$, and
\[K_{P, \v}(\z) =\frac{1}{(1-e^{-1-z_1})(1-e^{-1+z_2}) (1-e^{-1-z_3})} .\]

Summing over all the cones in $\N(P)$, we get
\begin{align*}
K_P 
&= 
    \sum_{\v\; \text{vertices}} K_{P,\v} \;\quad 
    - \; \sum_{e \;\text{edges} } K_{P,e} 
    &+& \sum_{F\; 2-\text{faces}} K_{P,F}  
    &-& 1\\
&= A_1 A_2 A_3 \quad \qquad  
    -  (A_1 A_2 + A_1 A_3 + A_2 A_3) 
    &+& (A_1 + A_2 + A_3) 
    &-& 1\\
&= (A_1 - 1)(A_2 - 1)(A_3 - 1).
\end{align*}
where
\[A_j = \frac{1}{1- e^{-1 - z_j}} +\frac{1}{1- e^{-1 + z_j}} \quad \text{for } j = 1,2,3. \]
Written as one fraction:
 \[K_P(\mathbf{z}) = \frac{(1 - e^{-2})^3}{(1 - e^{-1 + z_1})(1 - e^{-1 - z_1})(1 - e^{-1 + z_2})(1 - e^{-1 - z_2})(1 - e^{-1 + z_3})(1 - e^{-1 - z_3})}.\]
In general, let $P = [a,b]\times [c,d] \times [f,g]$. Then
    \[ K_P(\mathbf{z}) = \frac{(1 - e^{a-b})(1 - e^{c-d})(1 - e^{f-g})}{(1 - e^{-b + z_1})(1 - e^{a - z_1})(1 - e^{-d + z_2})(1 - e^{c - z_2})(1 - e^{-g + z_3})(1 - e^{f - z_3})}.\]

    Notice that since $P$ is the Cartesian product of three line intervals, the function $K_P$ can be factored into 
    \[K_P(\z) = K_{[a,b]}(z_1) \cdot K_{[c,d]}(z_2) \cdot K_{[f,g]}(z_3).\]
\end{ex}

We now justify our notation of the variable $\base$. Recall some definitions. For a polytope $P$ in $L\otimes_\Z\R$, the \emph{cone over $P$} is \[C(P):=\{(t,t\x)\:|\: t\in\R_{\geq0},\ \x\in P\}\subset \R\oplus (L\otimes_\Z\R).\] For a cone $C$ in $L'\otimes_{\Z}\R$, its \emph{dual cone} is \[C^*:=\{\v\in (L')^{\vee}\otimes_{\Z}\R\:|\: \v(\z)\geq0\text{ for all }\z\in C\}.\]
\begin{proposition}\label{prop: cone boundary vs cone}

Let $P\subset L\otimes_\Z\R$ be an integral polytope of full dimension. Then \[K_P(\base;\y)=\sum_{\v'\in\partial C(P)^* \cap (\Z\oplus L)^\vee}y_0^{-v_0}y_1^{-v_1}\cdots y_d^{-v_d} = (1-y_0^{-1})\sum_{\v' \in C(P)^* \cap (\Z \oplus L)^\vee}y_0^{-v_0}y_1^{-v_1}\cdots y_d^{-v_d}.\]
\end{proposition}
\begin{proof}
We start with the first equality. A point $\v'=v_0\oplus\v$ is on the boundary $\partial C(P)^*$ if and only if it satisfies $\min_{\x'\in C(P)} \v'(\x')=0$ where the minimum is attained by at least one point $\x' = (t,t\x) \in C(P) \setminus \{0\}$. This is equivalent to $v_0+\min_{\x\in P}\v(\x)=0$, i.e., $v_0=h_P(\v)$. Plugging in $v_0$, this matches with (\ref{eqn: K_P(y)}).

For the second equality, we are summing over the entire dual cone $C(P)^*$, so $v_0$ takes every integer value greater than or equal to $h_P(\v)$. 
\begin{align*} 
\sum_{\v' = (v_0,\v) \in C(P)^* \cap(\Z \oplus L)^\vee}y_0^{-v_0} \y^{-\v}
&= \sum_{\v\in (L')^\vee} \y^{-\v}\Big(\sum_{v_0=h_P(\v)}^\infty y_0^{-v_0} \Big)\\
&=  \frac{1}{1-y_0^{-1}}\sum_{\v\in (L')^\vee} \y^{-\v} y_0^{-h_P(\v)}\\
&= \frac{1}{1-y_0^{-1}} K_P(y_0; \y)\\
\end{align*}

\end{proof}

\begin{ex}
    Let $P = [1,2]$ in $\R$ and let $L$ be the standard lattice. The dual cone $C(P)^*$ is generated by $(-1,1)$ and $(2,-1)$. We can see in Figure \ref{fig:dual cone vs boundary example} that every horizontal (in $v_0$ direction) fiber of $C(P)^*$ is generated by a lattice point on the boundary $\partial C(P)^*$.

\begin{figure}[h!]
\centering

\begin{tikzpicture}[scale=1.0]
\node[right] at (1,1) {$1$};
\node[right] at (1,2) {$2$};
\node at (1,1) {$\bullet$};
\node at (1,2) {$\bullet$};
\node[left] at (1,1.5) {$P$};
\draw[thin](1,1)--(1,2);
\end{tikzpicture}
\qquad\qquad\qquad\qquad\qquad
\begin{tikzpicture}[scale=1.0]
\node[right] at (1,-0.5) {$-\frac{1}{2}$};
\node at (1,-0.5) {$\bullet$};
\node[left] at (1,0.5) {$P^\vee$};
\draw[thin](1,-0.5)--(1,2);
\end{tikzpicture}

\begin{tikzpicture}[scale=1.0]
\node at (0,0) {$\bullet$};
\node[below] at (1,1) {$(1,1)$};
\node[left] at (1,2) {$(1,2)$};
\draw[->,thick](0,0)--(2,2);
\draw[->,thick](0,0)--(2,4);
\draw[fill=black,opacity=0.2](0,0)--(2.5,2.5)--(2.5,5)--(0,0);
\end{tikzpicture}
\qquad
\begin{tikzpicture}[scale=1.0]
\node at (0,0) {$\bullet$};
\node at (1,0) {$\bullet$};
\node at (2,-1) {$\bullet$};
\node at (-1,1) {$\bullet$};
\node at (0,1) {$\bullet$};
\node at (1,1) {$\bullet$};
\node at (-2,2) {$\bullet$};
\node at (-1,2) {$\bullet$};
\node at (0,2) {$\bullet$};
\node at (2,0) {$\bullet$};

\node at (1,2) {$\dots$};
\node at (2,1) {$\dots$};
\node at (3,0) {$\dots$};
\node at (3,-1) {$\bullet$};
\node at (4,-1) {$\dots$};

\draw[->,thin](-2,0)--(4,0);
\node[below] at (4,0) {$v_0$};

\draw[->,thin](0,-2)--(0,3);
\node[right] at (0,3) {$v_1$};

\draw[->,thin, blue](-2,2)--(1,2);
\draw[->,thin, blue](-1,1)--(2,1);
\draw[->,thin, red](2,-1)--(4,-1);

\node[left] at (-1,1) {$(-1,1)$};
\node[below] at (2,-1) {$(2,-1)$};
\draw[->,thick](0,0)--(-1,1);
\draw[->,thick](0,0)--(3,-1.5);
\draw[fill=black,opacity=0.2](0,0)--(-2,2)--(3,-1.5)--(0,0);
\end{tikzpicture}
\caption{A polytope $P$, its polar dual $P^\vee$, the polytopal cones $C(P)$ and $C(P)^*$.}
\label{fig:dual cone vs boundary example}
\end{figure}
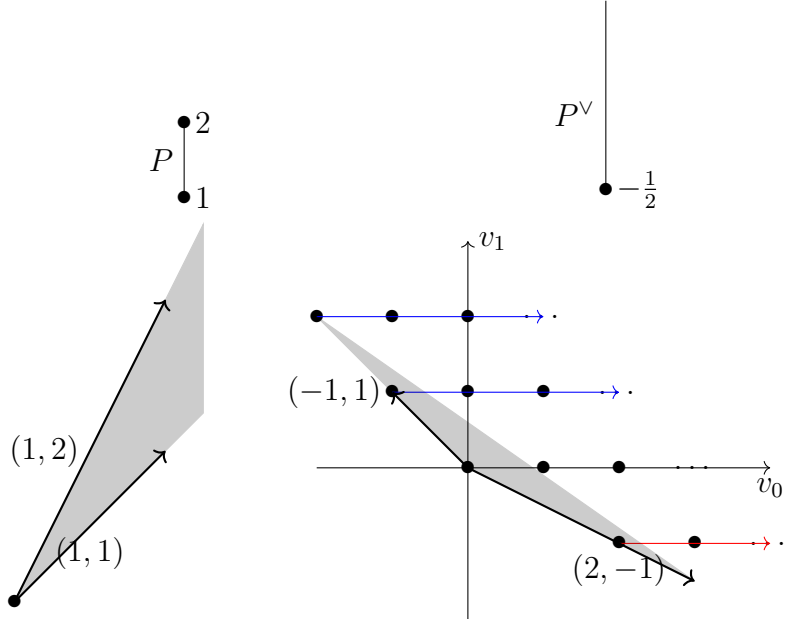

We have
\begin{align*}
    \sum_{\v' \in C(P)^*} y_0^{-v_0} y_1^{-v_1} &= \sum_{v_1 \in \Z} \sum_{v_0 = h_P(v_1)}^\infty y_0^{-v_0} y_1^{-v_1} \\
    &= \sum_{v_1 \leq 0} \sum_{v_0 = -2v_1}^\infty y_0^{-v_0} y_1^{-v_1} + \sum_{v_1 \geq 0} \sum_{v_0 = -v_1}^\infty y_0^{-v_0} y_1^{-v_1} -\frac{1}{1-y_0^{-1}}\\
    &= \frac{1}{1-y_0^{-1}} \underbrace{\Big(\sum_{v_1 \leq 0} y_0^{2v_1} y_1^{-v_1} + \sum_{v_1 \geq 0} y_0^{v_1} y_1^{-v_1} -1\Big)}_{= \sum_{\v'\in \partial C(P)^*}y_0^{-v_0} y_1^{-v_1}}\\
    &= \frac{1}{1-y_0^{-1}} \Big(  \frac{1}{1-y_0^{-2}y_1} + \frac{1}{1-y_0 y_1^{-1}} -1 \Big)\\
    &= \frac{1}{1-y_0^{-1}} K_{[1,2]}(y_0; y_1)
\end{align*}
by comparing with the formula in Example \ref{ex: 2-dim}.
\end{ex}

Comparing Proposition \ref{prop:lattice-point-generating-function} with the definition of the integer-point transform in \ref{eq: Def of A_c}, we see that the dual lattice function of $P$ and the integer-point transform of $C(P)^*$ determine each other.

\begin{cor}\label{cor: K_P vs A_C}
Let $P\subset L\otimes_\Z \R$ be in integral polytope of full dimension, then
\[K_P(y_0,\y) = (1-y_0^{-1})A_{C(P)^*} (y_0^{-1}, \y^{-1}). \]
\end{cor}

Together with Corollary \ref{reciprocity of A_C}, we have

\begin{cor}\label{cor: interior of cone}
    \[A_{\Int(C(P)^*)}(y_0, \y) = \frac{(-1)^{d+1}}{1-y_0^{-1}} K_{P}(y_0; \y).\]
\end{cor}

As shown in Proposition \ref{prop: cone boundary vs cone}, although by definition our $K_P$ function is a generating function on the boundary of the dual cone $C(P)^*$, it actually knows all the lattice points in the entire cone. Corollary \ref{cor: interior of cone} shows that, if we invert the input from $(y_0^{-1}, \y^{-1})$ to $(y_0; \y)$, $K_P$ outputs the generating function on the interior lattice points. The reciprocity here is a dual version of the Ehrhart-Macdonald Reciprocity, making $K_P$ a natural generating function to study for those interested in studying the dual Ehrhart theory. 

\subsection{Relation to Ehrhart series}
Let $Q \subset \R^d$ be a rational polytope.  The \emph{Ehrhart function}
$$
L(Q;m) := |mQ \cap \Z^d|
$$
of $Q$ is a quasipolynomial, known as the \emph{Ehrhart quasipolynomial} of $Q$.  The Ehrhart quasipolynomial has period dividing $r$, where $r \in \Z_{>0}$ is the minimal positive integer such that $rQ$ is a lattice polytope.  The Ehrhart series of $Q$ is defined to be
$$
E(Q;t):= 1 + \sum_{m \geq 1} L(Q;m) t^m = \frac{h^*_\Z(Q;t)}{(1-t^r)^{\dim Q + 1}}.
$$
Here, $h^*_\Z(Q;t)$ is known as the \emph{$h^*$-polynomial} of $Q$.

The following proposition states the connection between our $K_P$ function and the Ehrhart quasipolynomial (and $h^*$-polynomial) of the dual polytope. In general, it is difficult to determine these functions in Ehrhart theory for a rational polytope $Q$. In the case that $Q=P^\vee$ is a lattice polytope, the proposition could provide an easier way for computation, since the computation of $K_P$ only depends on lattice points.

\begin{proposition}\label{prop:Ehrhart}
Let $P$ be a lattice polytope containing the origin in its interior.  We have 
$$
E(P^\vee;t) = \frac{1}{1-t}K_P(1/t;1,1,\ldots,1).
$$
\end{proposition}
\begin{proof}
Suppose that $P$ is a lattice polytope containing the origin in its interior.  Then the slice $\{v_0 = 1\} \cap C(P)^*$ of the dual cone $C(P)^*$ is exactly the polar polytope $P^\vee$.  Similarly, the slice $\{v_0 = t\} \cap C(P)^*$ is exactly the dilation $t(P^\vee)$. Summing the lattice points on each slice is the same as the sum over the entire cone. Therefore
\[ E(P^\vee ,t) = \sum_{(v_0,\v) = \v' \in C(P)^*} t^{v_0} = A_{C(P)^*} (t,1,1,\dots,1). \]
By Corollary \ref{cor: K_P vs A_C}, this implies
\[ E(P^\vee ,t) = \frac{1}{1-t}  K_P(1/t; 1,1,\dots, 1). \qedhere \]
\end{proof}

\begin{example}
Let us consider the quadrilateral $P$ from \cref{ex: 2-dim}.  We have
$$
K_P(y_0,\y) = 
\frac{(1 - y_0^{-1}) \Big( 1 + y_0^{-1}(1 + y_1 + y_2) - y_0^{-2}(1 + y_1 + y_1/y_2) - y_0^{-3}y_1 \Big)}{(1-y_0^{-1}y_1^{-1}y_2)(1-y_0^{-1}y_1y_2)(1-y_0^{-1}y_2^{-1})(1-y_0^{-1}y_1y_2^{-1})}.
$$
By \cref{prop:Ehrhart}, we have
$$
E(P^\vee;t)= \frac{1}{1-t}\frac{(1 - t) \Big( 1 + 3t - 3t^2 - t^3 \Big)}{(1-t)^4}
= \frac{t^2+4t+1}{(1-t)^3}.$$
Indeed, $P^\vee$ is the lattice quadrilateral with vertices $(-1,1),(-1,-1),(1,-1),(0,1)$ and Ehrhart polynomial $L(P^\vee,t) = 3t^2+3t+1$.  This agrees with the general result that the Ehrhart and $h^*$-polynomials are related by the formula
$$
L(Q;t) = \sum_{i=0}^d h_i \binom{t-i+d}{d}
$$
where $h^*_{\Z}(Q;t) = h_0+h_1t+ \cdots + h_d t^d$.
\end{example}

\subsection{Relation to the dual volume function}
Recall from \cite[Definition 2.5]{GLX} the definition of dual volume functions.
\begin{definition}\label{def:dual-volume-function}
Let $P\subset L\otimes_\Z\R$ be a rational polytope of full dimension. Let $\T=\{C_1,\ldots,C_N\}$ be any triangulation of $\N({P})$ into full-dimensional simplicial cones. Each extremal ray in $\T$ gives an integral generator $\u_i$. The \emph{dual volume function} of $P$ is \[\VolC_{\z}(P):=\sum_{C=\mathrm{span}_{\R_{\geq0}}(\u_{j_1},\ldots,\u_{j_d})\in\T}\frac{\det|\u_{j_1},\ldots,\u_{j_d}|}{h_{P-\z}(\u_{j_1})\cdots h_{P-\z}(\u_{j_d})}.\]
\end{definition}
\begin{theorem}\label{thm:limit-to-VolC}
Let $P\subset L\otimes_\Z\R$ be a rational polytope of full dimension $d$. Then \[\lim_{\epsilon\rightarrow0}\frac{\epsilon^d}{d!} K_{P}(e^{\epsilon};\z)=\VolC_{\z}(P).\]
\end{theorem}
\begin{proof}
Let $\T$ be a complete, rational simplicial fan that refines $\N(P)$ (including the lower dimensional pieces) as in Theorem~\ref{thm: sum-of-cones}, which writes $K_{P}(\base;\z)$ as a sum over cones in $\T$. Let $C\in\T$ be of dimension $m\leq d$ with integral generators $\u_1,\ldots,\u_m$. A straightforward calculation from calculus shows that \begin{equation}\label{eq:limd}\left(\frac{\partial^m}{\partial \epsilon^m}\prod_{i=1}^m(1-e^{-\epsilon A_i})\right)\Bigg|_{\epsilon=0}=m!A_1A_2\cdots A_m,\quad \left(\frac{\partial^k}{\partial \epsilon^k}\prod_{i=1}^m(1-e^{-\epsilon A_i})\right)\Bigg|_{\epsilon=0}=0\text{ for }k<m.\end{equation}
Suppose that $\dim(C) = m < d$.
By Lemma~\ref{lem: parallelepiped formula}, the numerator of $\frac{\epsilon^d}{d!}K_{P,C}(e^{\epsilon};\z)$ is equal to $\epsilon^d\cdot\sum_{\v\in\Pi(C)\cap L^{\vee}}e^{\epsilon\v(\p-\z)}$, and evaluates to $0$ when taking the $m$-th derivative and then setting $\epsilon$ to $0$. By \eqref{eq:limd} taking the $m$-th derivative of the denominator and then setting $\epsilon$ to $0$ is nonzero.  By L’H\^opital’s Rule, we conclude that $\lim_{\epsilon \to 0} \frac{\epsilon^d}{d!}K_{P,C}(e^{\epsilon};\z) = 0$. 

Now suppose that $C$ is full-dimensional, i.e. $m=d$, by L’H\^opital’s Rule, \[\lim_{\epsilon\rightarrow0}\frac{\epsilon^d}{d!}K_{P,C}(e^{\epsilon};\z)=\lim_{\epsilon\rightarrow0}\frac{\epsilon^d\sum_{\v\in\Pi(C)\cap L^{\vee}}(e^{\epsilon \v(\p-\z)})}{d!\prod_{i=1}^d(1-e^{-\epsilon h_{P-\z}(\u_i)})}=\frac{|\Pi(C)\cap L^{\vee}|}{\prod_{i=1}^d h_{P-\z}(\u_i)}.\]
As $|\Pi(C)\cap L^{\vee}|=\det|\u_1,\ldots,\u_d|$, this is exactly the contribution of $C$ towards $\VolC_{\z}(P)$. The result follows. 
\end{proof}

\begin{remark}\label{rmk:Laplace}
The analogy here between \keydef and dual volume function is that of Laplace transform and discrete Laplace transform. Let $f$ be a function on $\R^d$. The \emph{Laplace transform} $\mathcal{L}(f)$ is defined by \[\mathcal{L}(f)(\z):=\int_{\R_{>0}^d}f(\x)\exp(-\langle\x,\z\rangle)d^d\x\]
and the \emph{discrete Laplace transform} is defined by \[\mathscr{L}(f)(\z):=\sum_{\x\in\Z^d_{\geq0}}f(\x)\exp(-\langle\x,\z\rangle).\]
Theorem~2.13 of \cite{GLX} provides a formula of the dual volume function $\VolC_{\z}(P)$ as a sum of the Laplace transform of the function $\exp(-h_P)$ over cones in $\N(P)$: \[\VolC_{\z}(P)=\int_{\R^d}\exp(-h_{P}(\x)-\langle\x,\z\rangle)d^d\x,\]
whereas our main definition (Definition~\ref{def:dual-lattice-point-generating-function}) is the discrete analogue. See \cite{lam2025matroidsamplitudes} for a discussion of these two transforms in the setting of scattering amplitudes. 
\end{remark}


\section{Properties of $K_P$}\label{sec:properties}
We establish several properties of the function $K_P$, with the eventual goal of showing that the rational differential form $\Theta_P:=K_P(\y)\omega_{T_P}$ is the canonical form of $\exp(P)$, which is the image of $P$ under the exponential function $\exp$, to be discussed in detail in Section~\ref{sec:toric}.
\subsection{The valuative property}\label{sub:valuation}
The following theorem shows that $K_P$ is a valuation, and in particular, the lower-dimensional terms do not contribute in the equality.

\begin{theorem}\label{thm:valuation}
Let $P_1,P_2,\ldots,P_r$ be polytopes such that $P_1,\ldots,P_a$ are full-dimensional, while the remaining $r-a$ polytopes are lower dimensional.  Suppose that $$\sum_{i=1}^r \alpha_i [P_i] = 0$$ in the polytope algebra. Then \[\sum_{i=1}^a \alpha_i K_{P_i}(y_0; \z) =0.\]
\end{theorem}

\begin{example}
Consider the polytope in Example~\ref{ex:dim-1}. Let $a\leq b\leq c\in\R$. Recall
\[K_{[a,b]}(z)=\frac{1}{1-\exp(z-b)}+\frac{1}{1-\exp(a-z)}-1.
\]
We have
\begin{align*}
& \quad K_{[a,b]}(z) + K_{[b,c]}(z)\\
&= \left(\frac{1}{1 - \exp(z-b)} + \frac{1}{1-\exp(a-z)} - 1\right) + \left(\frac{1}{1 - \exp(z-c)} + \frac{1}{1-\exp(b-z)} - 1\right) \\
&=\frac{1}{1 - \exp(z-c)} + \frac{1}{1-\exp(a-z)} + \frac{1}{1-\exp(z-b)} + \frac{-\exp(z-b)}{1-\exp(z-b)} - 2 \\
&= K_{[a,c]}(z).
\end{align*}
    
\end{example}




\begin{lemma}\label{lem:slice}
Suppose that $P$ is a full-dimensional polytope that has been subdivided into two full-dimensional polytopes $Q_1,Q_2$ by a single hyperplane $H$.  Then 
\begin{equation}\label{eq:PQ}
K_P(y_0; \y) = K_{Q_1}(y_0; \y) + K_{Q_2}(y_0; \y).
\end{equation}
\end{lemma}

\begin{proof} 
For simplicity we will assume in the proof below $y_0 = e$. We do not lose any generality since our arguments on normal fans are independent from the choice of $y_0$, and Definition \ref{def:dual-lattice-point-generating-function} defines $K_P(y_0; \y)$ as a valid rational function on $y_0$ and $\y$. In the rest of this proof, we use $K_P(\y)$ in place of $K_P(y_0; \y)$.

Let $H_{\geq 0}$ and $H_{\leq 0}$ be the two closed halfspaces such that $Q_{1} = H_{\geq 0} \cap P$ and $Q_{2} = H_{\leq 0} \cap P$. 
Let $F$ be any face of $P$, and let $F_1 = F\cap Q_1$, $F_2 = F\cap Q_2$. Notice that $F_1, F_2$ are (possibly empty) faces of $Q_1, Q_2$, whose dimensions depend on whether and how $F$ intersects with $H$. Let $C_F(P)$ be the normal cone of $F$ in the normal fan $\N(P)$. Notice that the normal fans of $Q_1$ and $Q_2$ together form a refinement of the normal fan of $P$.

To show the equality in (\ref{eq:PQ}), we will compare the decompositions from \cref{thm: sum-of-cones-faces} for both sides of the expression.

Partition the set of faces of $P$ into the following three categories:

\noindent\textbf{Case 1:} $F$ does not intersect with $H$. In this case $F$ is a face of exactly one of the two polytopes $Q_1, Q_2$. Suppose $F\in Q_1$, then $K_{P,F}(\y) = K_{Q_1, F}(\y)$. This term shows up exactly once on each side, so it cancels out.
    
\noindent\textbf{Case 2:} $F$ is contained in $H$. In this case $F= F_1 = F_2$, so $F$ is a face in all three polytopes. The terms that involve $F$ on both sides are $K_{P,F}(\y)$, $K_{Q_1, F}(\y)$, and $K_{Q_2,F}(\y)$. The relations of the normal cones are
\[C_{F}(P)  = C_{F}(Q_1) \cap C_{F}(Q_2); \]
\[C_{F}(Q_1\cap Q_2) = C_{F}(Q_1) \cup C_{F}(Q_2). \]
We have the equality of indicator functions:
\begin{equation}\label{eq:inclexcl} [C_F(P)] = [C{_{F}}(Q_1)] + [C{_{F}}(Q_2)] - [C{_{F}}(Q_1\cap Q_2)].\end{equation}

By \cref{def:KPF}, $K_{P,F}(\y)$ is a rational function associated with the normal cone $C_F(P)$, and therefore \eqref{eq:inclexcl} implies
\begin{align*} K_{P,F}(\y) &= K_{Q_1,F} (\y) +  K_{Q_2,F} (\y)- K'_{Q_1\cap Q_2, F}(\y)\\
&= K_{Q_1,F} (\y) +  K_{Q_2,F} (\y).
\end{align*}
Here, $K'_{Q_1 \cap Q_2,F}(\y)$ denotes \cref{def:KPF} applied to the cone $C_F(Q_1 \cap Q_2)$.  The second equality holds because $Q_1\cap Q_2$ is not full dimensional, and therefore the normal cone $C{_{F}}(Q_1\cap Q_2)$ contains a line (specifically, the line parallel to the normal vector of $H$), and hence $K'_{Q_1\cap Q_2, F} (\y)= 0$. Since the terms $K_{P,F}(\y)$, $K_{Q_1,F}(\y)$, and $K_{Q_2,F}(\y)$ each appears exactly once with the same sign on both sides of the equation we want, the identity above implies that these terms also cancel out. 

\noindent\textbf{Case 3:} The only case remaining is when $F$ intersects with $H$ but is not contained in $H$. We call such $F$ a ``cut face''. In this case, both $F_1, F_2$ are nonempty, and $F_1\neq F_2$. Let $F_0$ denote the intersection of $F_1, F_2$. 

If $H$ only cuts $F$ on its boundary, then in such case no new faces are created within $F$. Then exactly one part of it (say, this is $F_1$) equals $F$ while the other one ($F_2$) is some (lower dimensional) face contained in $H$, thus $F_1 =F$ and $F_2 = F_0$. We have $K_{P,F} = K_{Q_1, F_1}$, so these two terms cancel out. Notice that the $K_{Q_2, F_2}$ term was previously canceled in Case 2 since $F_2 = F_0$ is contained in $H$. 

After all the previous cancellations, the only remaining terms on the $Q_1$ and $Q_2$ side of the equation are those corresponding to the new faces created by the slicing of $H$. We now match them by considering our final case scenario, which is when $H$ cuts $F$ in its interior. In this case, it creates three new faces $F_1, F_2$, and $F_0$ in $Q_1$ and/or $Q_2$.
Notice that 
    \begin{align}\label{eq: cut face identity}
    K_{P,F}(\y) = K_{Q_1,F_1}(\y) = K_{Q_2, F_2}(\y).
    \end{align}
    
    Similar to Case 2, we can apply inclusion-exclusion on indicator functions, noting that 
\[C_{F_0}(Q_1) \cap C_{F_0}(Q_2)  = C_F(P),\]
\[C_{F_0}(Q_1) \cup C_{F_0}(Q_2) = C_{F_0}(Q_1\cap Q_2).\]
 This yields
\[K_{P,F}(\y) = K_{Q_1,F_0}(\y) + K_{Q_2,F_0}(\y) - K'_{Q_1\cap Q_2,F_0}(\y) = K_{Q_1,F_0}(\y) + K_{Q_2,F_0}(\y),   \]
where once again, we have used that the term  $K'_{Q_1\cap Q_2,F_0}(\y)$ vanishes since $C_{F_0}(Q_1\cap Q_2)$ contains a line.  Combining with (\ref{eq: cut face identity}), we get 
\[K_{P,F}(\y) = K_{Q_1,F_1}(\y) + K_{Q_2,F_2}(\y) - K_{Q_1,F_0}(\y) - K_{Q_2,F_0}(\y).\]

Since in this case $\dim(F) = \dim(F_1) = \dim(F_2) = \dim(F_0)+1$, the signs of the corresponding terms in the expression from \cref{thm: sum-of-cones-faces} match with the alternating signs here, therefore all terms cancel out perfectly. We have now taken care of every face on both sides, hence the equality $(\ref{eq:PQ})$ holds as desired.
\end{proof}

\begin{ex}\label{ex: subdiv}
Let $P$ be the polytope from Example \ref{ex: 2-dim}. Slice $P$ into two parts, $Q_1$ and $Q_2$, using the hyperplane $H$ given by $z_1 + z_2 +1 =0$. Label the four vertices of $P$ as $\v_1, \v_2, \v_3, \v_4$, see Figure~\ref{fig:subdiv-example}. Note that the normal cone of $P$ at a vertex $\v_i$ is the cone $C_i$ in Example \ref{ex: 2-dim}. In this example we explicitly denote them as $C_{v_i}(P)$ to distinguish them from the normal cones of $Q_1$ and $Q_2$. Let the new vertex created by $H$ be $\v_0 := (-1,0)$. 

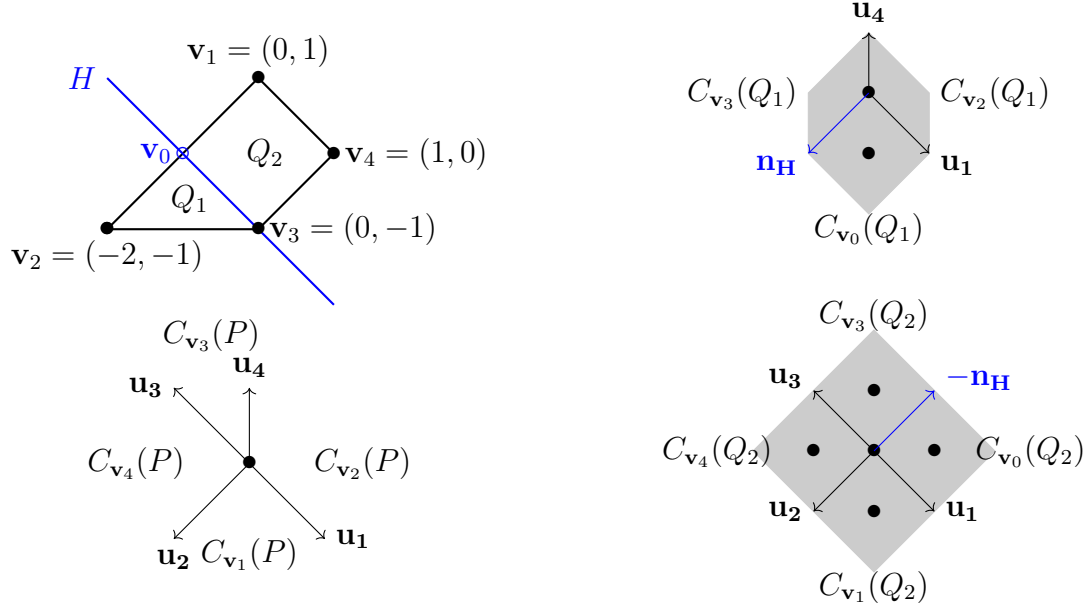
\begin{figure}[htbp]
\begin{minipage}[c]{0.45\textwidth}
\centering
\begin{tikzpicture}[scale=1.0]
\draw[thick](0,1)--(1,0)--(0,-1)--(-2,-1)--(0,1);
\draw[thick, blue](1,-2)--(-2,1);
\node at (0,1) {$\bullet$};
\node at (1,0) {$\bullet$};
\node at (-2,-1) {$\bullet$};
\node at (0,-1) {$\bullet$};
\node at (-2,-1) {$\bullet$};
\node[above] at (0,1) {$\v_1 =(0,1)$};
\node[below] at (-2,-1) {$\v_2= (-2,-1)$};
\node[right] at (0,-1) {$\v_3= (0,-1)$};
\node[right] at (1,0) {$\v_4 = (1,0)$};
\node[right] at (-0.3,0) {$Q_2$};
\node[left,blue] at (-2,1) {$H$};
\node[left,blue] at (-1,0) {$\v_0$};
\node[blue] at (-1,0) {$\circ$};
\node[right] at (-1.3,-0.6) {$Q_1$};

\end{tikzpicture}
\begin{tikzpicture}[scale=1.0]
\node at (0,0) {$\bullet$};


\draw[->,thin](0,0)--(-1,1);
\draw[->,thin](0,0)--(0,1);
\draw[->,thin](0,0)--(1,-1);
\draw[->,thin](0,0)--(-1,-1);

\node at (1.5,0) {$C_{\v_2}(P)$};
\node at (0,-1.2) {$C_{\v_1}(P)$};
\node at (-0.5,1.7) {$C_{\v_3}(P)$};
\node at (-1.5,0) {$C_{\v_4}(P)$};

\node[right] at (1,-1) {$\mathbf{u_1}$};
\node[below] at (-1,-1) {$\mathbf{u_2}$};
\node[left] at (-1,1) {$\mathbf{u_3}$};
\node[above] at (0,1) {$\mathbf{u_4}$};
\end{tikzpicture}
\end{minipage}
\hfill
\begin{minipage}[c]{0.5\textwidth}
\centering
\begin{tikzpicture}[scale=0.8]
\draw[fill=black,opacity=0.2](1,-1)--(1,0)--(0,1)--(-1,0)--(-1,-1)--(0,-2)--(1,-1);
\node at (0,0) {$\bullet$};
\node at (0,-1) {$\bullet$};

\draw[->,thin](0,0)--(0,1);
\draw[->,thin](0,0)--(1,-1);
\draw[->,thin, blue](0,0)--(-1,-1);

\node[right] at (1,0) {$C_{\v_2}(Q_1)$};
\node[below] at (0,-1.8) {$C_{\v_0}(Q_1)$};
\node[left] at (-1,0) {$C_{\v_3}(Q_1)$};

\node[right] at (1,-1.2) {$\mathbf{u_1}$};
\node[left, blue] at (-1,-1.2) {$\mathbf{n_H}$};
\node[above] at (0,1) {$\mathbf{u_4}$};
\end{tikzpicture}
\vspace{1em} 
        

\begin{tikzpicture}[scale=0.8]
\draw[fill=black,opacity=0.2](1,-1)--(2,0)--(1,1)--(0,2)--(-1,1)--(-2,0)--(-1,-1)--(0,-2)--(1,-1);
\node at (0,0) {$\bullet$};
\node at (0,-1) {$\bullet$};
\node at (-1,0) {$\bullet$};
\node at (1,0) {$\bullet$};
\node at (0,1) {$\bullet$};

\draw[->,thin](0,0)--(-1,1);
\draw[->,thin](0,0)--(1,-1);
\draw[->,thin,blue](0,0)--(1,1);
\draw[->,thin](0,0)--(-1,-1);

\node[below] at (0,-1.8) {$C_{\v_1}(Q_2)$};
\node[right] at (1.5,0) {$C_{\v_0}(Q_2)$};
\node[above] at (0,1.8) {$C_{\v_3}(Q_2)$};
\node[left] at (-1.5,0) {$C_{\v_4}(Q_2)$};

\node[right] at (1,-1) {$\mathbf{u_1}$};
\node[left] at (-1,-1) {$\mathbf{u_2}$};
\node[right, blue] at (1,1.2) {$\mathbf{-n_H}$};
\node[left] at (-1,1.2) {$\mathbf{u_3}$};
\end{tikzpicture}
\end{minipage}
\caption{Subdivision of polytope $P$ into $Q_1$ and $Q_2$, the normal fans $\N(P)$, $\N(Q_1)$, and $\N(Q_2)$. The shadowed area represents the fundamental parallelepipeds for each maximal cone.}
\label{fig:subdiv-example}
\end{figure}

We compute each of $K_{Q_1} (\z)$ and $K_{Q_2} (\z)$ using the formula from Lemma~\ref{lem: parallelepiped formula}. For $Q_1$, the three maximal cones in $\N(Q_1)$ are denoted as $C_{\v_0}, C_{\v_2}, C_{\v_3}$.  We use the generating rays $\u_1, \u_4$ from $\N(P)$ together with $\n_H = (-1,-1)$, a normal vector to $H$. 



\begin{eqnarray*}
K_{Q_1}(\z)
&=& K_{Q_1,\v_0} + K_{Q_1,\v_2} + K_{Q_1,\v_3}   - K_{Q_1,\v_0 \v_2} - K_{Q_1,\v_2\v_3} - K_{Q_1,\v_0\v_3} +  1\\
&=& 
\frac{1 + e^{z_2}}{(1 - e^{1+z_1+z_2})(1 - e^{-1-z_1+z_2})} 
+ \frac{1}{(1 - e^{-1-z_2})(1 - e^{-1-z_1+z_2})} 
+ \frac{1}{(1 - e^{-1-z_2})(1 - e^{1+z_1+z_2})} \\
&&\quad 
- \frac{1}{1 - e^{-1-z_1+z_2}} 
- \frac{1}{1 - e^{-1-z_2}} 
- \frac{1}{1 - e^{1+z_1+z_2}} 
+ 1
\end{eqnarray*}

For $Q_2$, since this is a square centered at the origin, the terms that show up in the formula are very symmetric. The generating rays we use are $\u_1, \u_2, \u_3$ from $\N(P)$, together with the normal vector to $H$, $-\n_H$. 

\begin{eqnarray*}
K_{Q_2}(\mathbf{z}) &=& K_{Q_2,\v_0} + K_{Q_2,\v_3} + K_{Q_2,\v_4} + K_{Q_2,\v_1}\\
&& \quad - K_{Q_2,\v_0 \v_3} - K_{Q_2,\v_3 \v_4} - K_{Q_2,\v_1 \v_4} - K_{Q_2,\v_0 \v_1} + 1 \\
&=&
\frac{1 + e^{-1-z_1}}{(1 - e^{-1-z_1+z_2})(1 - e^{-1-z_1-z_2})} 
+ \frac{1 + e^{-1-z_2}}{(1 - e^{-1-z_1-z_2})(1 - e^{-1+z_1-z_2})} \\
& & \quad +\frac{1 + e^{-1+z_1}}{(1 - e^{-1+z_1+z_2})(1 - e^{-1+z_1-z_2})} 
+ \frac{1 + e^{-1+z_2}}{(1 - e^{-1+z_1+z_2})(1 - e^{-1-z_1+z_2})} \\
& & \quad - \frac{1}{1 - e^{-1-z_1-z_2}} 
- \frac{1}{1 - e^{-1+z_1-z_2}} 
- \frac{1}{1 - e^{-1+z_1+z_2}} 
- \frac{1}{1 - e^{-1-z_1+z_2}} 
+ 1.
\end{eqnarray*}

One can manually verify that the sum of the two rational functions above, $K_{Q_1}+ K_{Q_2}$, indeed coincides with the formula for $K_P$ we computed earlier in Example~\ref{ex: 2-dim}. However, it is more insightful to see how the cases from the proof of Lemma \ref{lem:slice} dictate the cancellations here. Let's check some example faces.

\noindent\textbf{Case 1.} Consider the vertex $\v_4 $ of $P$. It lies entirely in $Q_2$ and is disjoint from $Q_1$. Looking at the normal fans, the two cones $C_{\v_4}(P) = C_{\v_4}(Q_2)$. Thus 
\[K_{P,\v_4} =\frac{1 + e^{-1+z_1}}{(1 - e^{-1+z_1+z_2})(1 - e^{-1+z_1-z_2})}  = K_{Q_2,\v_4}.\] (note that we use the equivalent notation $K_{P,C_4}$ to represent $K_{P,\v_4}$ in Example~\ref{ex: 2-dim}). And so they match across the equation $K_P = K_{Q_1} + K_{Q_2}$.

\noindent\textbf{Case 2.} Consider $\v_3$. This vertex lies inside $H$, which makes it a face of $P$, $Q_1$, and $Q_2$. By inclusion-exclusion on the normal cones, the relationship is given by 
\begin{align*}
   & K_{P, \v_3} - K_{Q_1, \v_3} - K_{Q_2, \v_3} \\
   =& \frac{1}{(1{-}e^{-1-z_2})(1{-}e^{-1+z_1-z_2})} 
   {-}\frac{1}{(1 {-} e^{-1-z_2})(1 {-} e^{1+z_1+z_2})}
   {-}\frac{1 {+} e^{-1-z_2}}{(1 {-} e^{-1-z_1-z_2})(1 {-} e^{-1+z_1-z_2})}\\
   =&0.
\end{align*}
Notice that cone $C_{\v_3}(Q_1 \cap Q_2)$, which is the union of $C_{\v_3}(Q_1)$ and $C_{\v_3}(Q_2)$, contains the line generated by the normal vector $\n_H$ to $H$, hence its corresponding term is $0$.

\noindent\textbf{Case 3.} The edge $\v_1 \v_2$ in $P$ is cut through by the hyperplane $H$ in its interior point $\v_0$, so $H$ splits it into two new edges, $\v_0 \v_2$ in $Q_1$ and $\v_0 \v_1$ in $Q_2$. The sum of the functions for these new sub-faces balances perfectly with the function of the original uncut edge in $P$. The relationship is
\[K_{P, \v_1 \v_2}(\mathbf{z}) = K_{Q_1, \v_0 \v_2}(\mathbf{z}) + K_{Q_2, \v_0\v_1} (\mathbf{z}) - K_{Q_1, \v_0}(\mathbf{z}) - K_{Q_2, \v_0}(\mathbf{z}).\]
We omit the computational verification here.
\end{ex}

\begin{proof}[Proof of \cref{thm:valuation}]
Let us consider the hyperplane arrangement consisting of all supporting hyperplanes of facets of all the polytopes $P_1,\ldots,P_a$.  Each polytope $P_i$ is a union of chambers in this hyperplane arrangement.  By repeatedly applying \cref{lem:slice}, we obtain the identity of rational functions
$$
K_{P_i}(\y) = \sum_{R \subset P_i} K_R(\y)
$$
where the summation is over chambers $R$ contained in $P_i$.  Substituting this into $\sum_i \alpha_i K_{P_i}(\y)$, the assumption implies that the coefficient of each $K_{R}(\y)$ sums to 0, giving the desired result.
\end{proof}

\subsection{The degree property}
We show that the degree of $K_P(\y)$ is negative in every direction in a precise sense. Specifically, let $\w \in \R^d$ be a nonzero vector. Define the \emph{$\w$-degree} of the monomial $\y^\p = y_1^{p_1} y_2^{p_2} \dots y_d^{p_d}$ to be
$$
\deg_\w(\y^\p) := \ip{\w, \p}.
$$
For any polynomial $f(\y)$, the $\w$-degree $\deg_\w(f(\y))$ is the maximum $\w$-degree of the monomials in it. And for a rational function $f(\y)/g(\y)$ in the variables $\y$, we declare that $$\deg_\w(f(\y)/g(\y)) = \deg_\w(f(\y)) - \deg_\w(g(\y)).$$

\begin{proposition}\label{prop:deg}
Let $P$ be a non-empty full-dimensional rational polytope in the lattice $L \subset L \otimes_\Z \R$, and let $\w \in \R^d$ be nonzero.  Then
$$
\deg_\w( K_P(\base; \y)) < 0.
$$

\end{proposition}

We start with some building blocks of the proof. Let $\F$ be a complete simplicial fan.  For each cone $C = \sp_{\geq 0}(\v_1,\ldots,\v_r)$, the set of \emph{fundamental points} is $\Int(\Pi(C)) \cap L^\vee$, where $\Pi(C)$ denotes the fundamental parallelepiped of $C$.  Here, the interior $\Int(\Pi(C))$ is taken in the linear span of $C$.

\begin{defn}
The \emph{fundamental star} of a simplicial fan $\F$ is the set of all fundamental points of cones of $\F$:
\[\star_\F := \bigsqcup_{C \in \F} \Int(\Pi(C)) \cap L^\vee\] 
\end{defn}

Let $\Delta$ be a full-dimensional rational simplex with normal fan $\F$.  Define the monomial
$$
\bby^\p := y_0^{-h_\Delta(\p)} y_1^{p_1} \cdots y_d^{p_d},
$$
and for each $\v_i$ a generating ray of $\F$, we define $\h_i:= \bby^{\v_i}$.

Let $\q \in \star_\F$ be a fundamental point and suppose that $\q \in \Int(\Pi(C))$ where $C = \sp_{\geq 0}(\v_1,\ldots,\v_r)$. Define the rational function
$$
K_\q(y_0; \y) := \bby^\q \frac{1 - \prod_{i=r+1}^{d+1} \h_i }{\prod_{i=1}^{d+1} (1-\h_i)},
$$
where $\v_{r+1},\ldots,\v_{d+1}$ are the remaining generating rays of $\F$.  

\begin{proposition}\label{prop:Kq}
Let $\F$ be the normal fan of a full-dimensional rational simplex $\Delta$.  Then we have $$K_{\Delta}(y_0; \y) = \sum_{\q \in \star_\F} K_\q(y_0; \y).$$
\end{proposition}
\begin{proof}
Let $\q \in \star_\F$ be a fundamental point in $\Int(\Pi(C_\q))$ where $C_\q = \sp_{\geq 0}(\v_1,\ldots,\v_r)$.  Then $\q$ belongs to $\Pi(C) \cap L^\vee$ for every cone $C \in \F$ that contains $\q$.  These cones are of the form $C = \sp_{\geq 0}(\v_1,\ldots,\v_r,\v_{j_1}, \ldots, \v_{j_k})$
where $J=\{j_1,\ldots,j_k\} \subsetneq [\v_{r+1},\ldots,\v_{d+1}]$.  Thus
\begin{align*}
K_\Delta(y_0; \y) &=
\sum_{C \in \F} (-1)^{\codim(C)} \sum_{\p \in C} \bby^\p  \\
&= \sum_{C \in \F} (-1)^{\codim(C)} \frac{\sum_{\q \in \Pi(C) \cap L^\vee}\bby^\q}{\prod (1 - \bby^{\v})}\\ 
&= \sum_{\q \in \star_\F} \bby^\q \sum_{C \ni \q} (-1)^{\codim(C)} \frac{1}{\prod  (1 - \bby^{\v})},
\end{align*}
where the products in the denominators are over the generators $\v$ of $C$.  For a fixed $\q$,
\begin{align*}
\sum_{C \ni \q} (-1)^{\codim(C)} \frac{1}{\prod_i  (1 - \bby^{\v})}
&= (-1)^{d-r} \prod_{i=1}^{r} \frac{1}{1-\h_i} \sum_{J \subsetneq [r+1,d+1]}(-1)^{|J|} \prod_{j \in J} \frac{1}{1-\h_j} \\
&= (-1)^{d-r} \prod_{i=1}^{d+1} \frac{1}{1-\h_i} \sum_{J \subsetneq [r+1,d+1]}(-1)^{|J|} \prod_{j \in [r+1,d+1] \setminus J} (1-\h_j) \\
&= \prod_{i=1}^{d+1} \frac{1}{1-\h_i} \left(1 - \prod_{i=r+1}^{d+1} \h_i \right).
\end{align*}
Summing over $\q$, we obtain the claimed result.
\end{proof}

We are now ready to prove \Cref{prop:deg}.
\begin{proof}[Proof of \Cref{prop:deg}]
In this proof we fix $\w$ and write $\deg$ for $\deg_\w$.
We will also assume below that $P$ lies in the standard lattice ($y_0 = e$) and write $K_P(\y)$ for $K_P(e; \y)$ since the choice of $y_0$ does not affect the degree. By definition, we have that $\deg(f+g) \leq \max(\deg(f),\deg(g))$, so by \cref{thm:valuation} and triangulating $P$, we can reduce the problem into showing that $\deg K_\Delta(\y) < 0$ for every lattice simplex $\Delta$.  Furthermore, by \cref{prop:Kq}, it suffices to show that $\deg(K_\q) < 0$ for any $\q \in \star_\F$, where $\F$ is the normal fan of a lattice simplex $\Delta$.  Write $\q = \sum_{i=1}^r \alpha_i \v_i$ for $\alpha_i \in (0,1)$.  We compute
\begin{align*}
&\deg(K_\q(\y)) \\&= \deg(\y^\q) + \max(0, \sum_{i=r+1}^{d+1} \deg(\h_i)) - \sum_{i=1}^{d+1} \max(0,\deg(\h_i)) \\
&= \ip{\w,\sum_{i=1}^r \alpha_i \v_i} +  \max(0, \ip{\w,\sum_{i=r+1}^{d+1} \v_i})  -  \sum_{i=1}^{d+1} \max(0,\ip{\w,\v_i}) \\
&=  \left(\sum_{i=1}^{r}\Big(\alpha_i \ip{\w,\v_i} -  \max(0,\ip{\w,\v_i})\Big) \right) + \left(\max(0, \ip{\w,\sum_{i=r+1}^{d+1} \v_i}) -  \sum_{i=r+1}^{d+1} \max(0,\ip{\w,\v_i})\right).
\end{align*}
Both terms are $\leq 0$, hence $\deg(K_\q(\y)) \leq 0$. Suppose $\deg(K_\q(\y)) = 0$, then this will force $\langle \w, \v_i \rangle = 0$ for all $1\leq i \leq r$, which is impossible since $\w$ is nonzero. 
Thus $\deg K_\q(\y) < 0$. 
\end{proof}

\cref{prop:deg} can be reformulated as follows.
\begin{cor}\label{cor:deg}
Let $P$ be a non-empty full-dimensional rational polytope.  Write $K_P(\base;\y) = f(\y)/g(\y)$ as a ratio of polynomials in $y_1,y_2,\ldots,y_d$.  Then the Newton polytope of  $f(\y)$
 is contained in the interior of the Newton polytope of $g(\y)$.   
\end{cor}

\begin{ex}
    Consider the polytope $P$ from \cref{ex: 2-dim}.  We have $K_P(\base; \y) = f(y_1,y_2)/g(y_1,y_2)$ where the Newton polytopes are given by
    \begin{align*}
N(f) &= \conv((0,0),(1,0),(0,1),(1,-1)) \\ N(g) &= \conv((0,0),(-1,-1))+ \conv((0,0),(1,1)) \\&+\conv((0,0),(0,-1))+\conv((0,0),(1,-1)).
\end{align*}
\cref{cor:deg} can be verified directly.
\end{ex}

Recall that $\y$ can be viewed as coordinates on a torus $(\C^\times)^d$, which may be compactified to $\Pbb^d$.  We explicate one consequence of \cref{prop:deg}, which is also part of \cref{thm:expP2}.

\begin{cor}\label{cor: simple pole}
    Let $P$ be a full-dimensional rational polytope. The rational differential form $ \Theta_P := K_P(y_0; \y) \omega_{T_P}$ has no poles along any of the coordinate hyperplanes in $\Pbb^d \setminus (\C^\times)^d$.
\end{cor}
\begin{proof}
    There are $d+1$ coordinate hyperplanes: the $d$ hyperplanes $\{y_1 = 0\}, \ldots, \{y_d = 0\}$, and the hyperplane at infinity. The claim follows from applying \cref{prop:deg} with $\w = -e_1, \ldots, \w = -e_d$ and finally $\w = e_1 + e_2 + \cdots + e_d$.
\end{proof}

\subsection{Residues on facets}
We now show that $\Theta_P$ in Corollary \ref{cor: simple pole} has the correct residue $\Res_{T_F}$ for each facet $F$ of $P$. By Theorem \ref{thm: sum-of-cones}, it suffices to consider one cone of the normal fan $\mathcal{N}(P)$ at a time. Moreover, we can assume that all cones are simplicial, since any non-simplicial cone can be triangulated without changing $K_{P}$.


Recall that $A_C(\t)$ is the integer-point transform of a cone $C$ as defined in (\ref{eq: Def of A_c}). We first show that ``the residue'' of $A_C(\t)$ along the subspace spanned by a facet coincides with the integer-point transform on that facet.

Let $\u_1,\ldots,\u_m$ be integral generators of the rational simplicial cone $C$, and let $F$ be the facet of $C$ that does not contain $\u_1$. Up to a ${\rm GL}(d,\Z)$-change of coordinates, we may assume that $\u_1 = \e_1$, and so $\t^{\u_1} = t_1$. We define the residue of $A_C$ along $t_1=1$ as
\[ \Res_{t_1=1} A_C(\t) : = \Big[(1-t_1) \cdot A_C(\t)\Big]\Big|_{t_1=1}.\]

Let $H' = \R^d/\u_1$ be the linear subspace orthogonal to $\u_1 =\e_1$, and let $L' = L\cap H'$. Define $\tilde{F}$ as the image of $F$ in $H'= \R^d/\e_1$, hence it is associated with the lattice $L'$. The lattice point generating function of $\tilde{F} \subset L' \otimes_\Z \R$ is therefore in variables $\tilde{\t} := (t_2, \dots, t_d)$.

\begin{proposition}\label{prop: A_C Residue}
\[ 
\Res_{t_1 = 1} A_C(\t)= A_{\tilde{F}} (\tilde{\t})
\]
\end{proposition}

\begin{proof}
We have

\begin{eqnarray}\label{eqn: Res}
\Res_{t_1 = 1} A_C(\t)  &=&  (1-t_1) A_C(\t)|_{t_1=1}.\\  \nonumber
&=& \Big(\sum_{\v\in \Pi(C) \cap L^\vee} \t^\v \Big) \prod_{i=2}^m \frac{1}{1-\t^{\u_i}} \Big|_{t_1=1} \quad (\text{by Prop. } \ref{prop:lattice-point-generating-function} )\\ \nonumber
&=& \Big(\sum_{\v \in S} \t^\v \Big)\Big|_{t_1=1},
\end{eqnarray}
where 
\begin{align*}
S &=  \{\v \in L^\vee\; |\; \v \in \Pi(C)+ \sum_{i=2}^m q_i \u_i \text{ with } q_i\in \Z_{\geq 0}\}\\
  &= \{\v \in C\;|\; \v-\u_1 \notin C\}.
\end{align*}

Each $\v \in S$, under the canonical projection to $H'$,  is a lattice point $\tilde{\v} \in \tilde{F}\cap L'$ and 
\[ \tilde{\t}^{\tilde{\v}} = \t^\v |_{t_1=1}.\]
Therefore (\ref{eqn: Res}) equals the following:

\begin{align*}
 \Big(\sum_{\v \in S} \t^\v \Big)\Big|_{t_1=1}
&= \sum_{\tilde{\v} \in \tilde{F} \cap L'} \tilde{\t}^{\tilde{\v}} = A_{\tilde{F}} (\tilde{\t}),
\end{align*}
as desired.

\end{proof}

\begin{example}\label{ex: compute A_C}
Let $L$ be the standard lattice $\Z^2$, and let $C$ be the cone spanned by $\u_1 = \e_1$ and $\u_2 = (2,3)$. By definition, we have
\[A_C(\t) = \frac{(1+t_1t_2 + t_1^2t_2^2)}{ (1-t_1)(1-t_1^2t_2^3)}. \]

Therefore 
\[ \Res_{t_1=1} A_C(\t) = \big[(1-t_1)\cdot A_C(\t)\big]\Big|_{t_1=1} = \frac{1+t_2+t_2^2}{1-t_2^3} = \frac{1}{1-t_2}.\]

Let $F$ be the facet of $C$ spanned by $\u_2$, and let $\tilde{F}$ be its projection onto $H'$, so
\[A_{\tilde{F}}(t_2) = \frac{1}{1-t_2}. \]
This agrees with \cref{prop: A_C Residue}.
\end{example}

\begin{figure}[h!]
\centering

\begin{tikzpicture}[scale=1.0]
\draw[fill=black,opacity=0.2](1,0)--(3,3)--(2,3)--(0,0);
\node at (0,0) {$\bullet$};
\node at (1,1) {$\bullet$};
\node at (2,2) {$\bullet$};

\node at (0,0) {$\bullet$};
\node at (0,1) {$\circ$};
\node at (0,2) {$\circ$};
\node at (0,3) {$\circ$};

\draw[->,thin](0,0)--(2,3);
\draw[->,thin](0,0)--(1,0);

\draw[-,thin](0,0)--(0,3.5);

\node at (-0.5,1.5) {$\tilde{F}$};
\node at (2.5,1) {$\Pi(C)$};

\node[right] at (1,0) {$\mathbf{u_1}$};
\node[below] at (2.5,3.5) {$\mathbf{u_2}$};
\end{tikzpicture}
\caption{Fundamental parallelepiped in Example \ref{ex: compute A_C}}\label{fig:dual-polytope-example}
\end{figure}
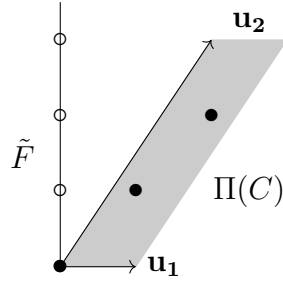

We will now use Proposition \ref{prop: A_C Residue} to prove the residue condition on $K_P$. 


\begin{theorem}\label{thm:ResF} Let $P$ be a polytope in $L\otimes_\Z \R$ and $F$ be a facet of $P$, we have
    \[\Res_{F} K_P(y_0; \y) \omega_{T_P} = K_F (y_0; \y) \omega_{T_F}.\]
\end{theorem}

\begin{proof}
Let $\F$ be the normal fan of $P$ in its affine span.  By Theorem \ref{thm: sum-of-cones}, we have
\begin{eqnarray}\label{eq: Res sum}
    \Res_{T_F} K_P(y_0; \y) \omega_{T_P}  
&=\Res_{T_F} \sum_{C\in \F} (-1)^{\mathrm{codim}(C)}K_{P,C}(y_0; \y) \omega_{T_P}  \nonumber \\
&= \sum_{C\in \F} (-1)^{\mathrm{codim}(C)} \Res_{T_F} K_{P,C}(y_0; \y) \omega_{T_P}.
\end{eqnarray}
Since we are free to triangulate the cones in $\F$, we may reduce to the case that each $C \in \F$ is simplicial. Since $F$ is a facet of $P$, its dual cone $C_F \in \F$ is generated by one vector, say, $\u_F$. 
Notice that for each $C\in \F$, if $\u_F \notin C$, then the residue $\Res_{T_F} K_{P,C}(y_0; \y) \omega_{T_P} = 0$. Therefore the sum in (\ref{eq: Res sum}) is reduced to summing over the collection
\[\F' = \{C\in \F\; |\; \u_F \in C\}.\] 

Let $H_F$ be the affine span of $F$ and $\F_F$ be the normal fan of $F$ in $H_F$. This fan can be identified with the image of $\F'$ in $L^\vee / \u_F \otimes_\Z \R$, which we denote by $\tilde{\F}$. For each cone $C\in \F'$, we write its image  in $\tilde{\F}$ as $\tilde{C}$. We'll show that
\[\Res_{T_F} K_{P,C} (y_0; \y) \omega_{T_P} = K_{F,\tilde{C}} (y_0; \y) \omega_{T_F}. \]

Up to coordinate changes and shifting in $\y$, we may assume that $\u_F = \e_1$ and $T_F$ is given by $y_1=1$. 
By Lemma \ref{lem: parallelepiped formula} and Proposition \ref{prop:lattice-point-generating-function}, we have
\begin{align*}
    \Res_{T_F} K_{P,C} (y_0; \y) \omega_{T_P} &= \Res_{t_1 = 1} A_C(\t) \omega_{T_F} 
\end{align*}
where $\t  = (t_1, \dots, t_d)= \exp_{y_0}(\p - \z) = \exp_{y_0}(\p) \cdot \y^{-1}$ for any point $\p\in F$, and $\omega_{T_F} = \frac{d y_2}{y_2} \wedge \dots \wedge \frac{d y_d}{y_d}$. By Proposition \ref{prop: A_C Residue} we have
\begin{align*}
    \Res_{t_1 = 1} A_C(\t) \omega_{T_F}  &= A_{\tilde{F}} (\tilde{\t}) \omega_{T_F}  \\
    &= K_{F,\tilde{C}} (y_0; \tilde{\y}) \omega_{T_F},
\end{align*}
where $\tilde{\y} = (y_2, \dots ,y_d)$.
Therefore
\begin{eqnarray*}
    \Res_{T_F} K_P(y_0; \y) \omega_{T_P}  
&=& \sum_{C\in \F'} (-1)^{\mathrm{codim}(C)} \Res_{T_F} K_{P,C}(y_0; \y) \omega_{T_P}\\
&=& \sum_{\tilde{C}\in \tilde{\F}} (-1)^{\mathrm{codim}(\tilde{C})} K_{F,\tilde{C}} (y_0; \tilde{\y}) \omega_{T_F}\\
&=& K_{F} (y_0; \tilde{\y}) \omega_{T_F},
\end{eqnarray*}
as desired.
\end{proof}

\section{Toric polytopes}\label{sec:toric}
\def\L{{\mathbf{L}}}
\def\M{{\mathbf{M}}}
\def\bm{{\mathbf{m}}}
\def\bell{{\mathbf{l}}}
\def\Hom{{\rm{Hom}}}

In this last section, we introduce \emph{exponential polytopes} and show that they are positive geometries. For a lattice $L$ as in previous sections, in this section we will consider an embedding $L \hookrightarrow \Z^m$ of $L$
 into a fixed lattice $\Z^m$.  This choice manifests itself as a choice of the compactification $\Pbb^m$ of the torus $(\C^\times)^m$.

\subsection{More on lattices and toric varieties}\label{sec:toricsub}

A sublattice $L$ in  $\Z^m$ is called \emph{saturated} if $L = (L \otimes_\Z \Q) \cap \Z^m$.  In this case, the quotient $\Z^m/L$ is also free abelian. Let $L$ be a saturated sublattice in  $\Z^m$.  Associated to $L$ is the inclusion of tori
$$
T_L := L \otimes_\Z \C^\times \cong (\C^\times)^d \hookrightarrow (\C^\times)^m =: T.
$$
The lattices $L$ and $\Z^m$ are the cocharacter lattices of $T_L$ and $T$ respectively.

Let $\L$ be a $d \times m$ matrix whose rows correspond to integer vectors that generate $L$.  Let $\bell_1,\ldots,\bell_m$ denote the columns of $\L$.  Then we obtain an explicit description of the inclusion $T_L \hookrightarrow T$ as the monomial map
$$
(\C^\times)^d \to (\C^\times)^m, \qquad \b =(b_1,b_2,\ldots,b_d) \mapsto (\b^{\bell_1},\ldots,\b^{\bell_m}).
$$

The torus $T_L$ acts on $T$ and this extends to an action on the compactification $\Pbb^m$ of $T$.  Note that the compactification $\Pbb^m$ depends on a choice of basis of $\Z^m$, which we assume to be fixed.

For a point $p \in T$, we call the orbit $T_L \cdot p$ a \emph{shifted torus}.

\begin{definition}
For a point $p \in T$, the \emph{toric variety} $X(L,p)$ associated to the pair $(L,p)$ is the normalization of the closure $\overline{T_L \cdot p}$ of the orbit $T_L \cdot p$ inside $\Pbb^m$.
\end{definition}

When $p = e$ the identity, we denote $X(L) = X(L,e)$.  Note that the action of $T_L$ on $T$ or on $\Pbb^m$ maps any shifted subtorus or projective toric variety to its corresponding variety passing through the identity $e \in T$.  Thus the isomorphism class of these varieties does not depend on the choice of $p$.  In particular, the choice of the point $p$ only matters when we are interested in the explicit projective embedding of $X(L,p)$.

Let $\L$ be the matrix associated with $L$ defined above. We can define a lattice polytope $Q_L$ associated with $L$ by taking the convex hull of the origin and the lattice points $\bell_1,\ldots,\bell_m \in \Z^d$ (the columns of $\L$).
The orbits of the action $T_L$ on $X(L,p)$ are in bijection with the faces of $Q_L$ (see \cite[Theorem 3.2.6]{Cox-Little-Schenck}).

The map $(\C^\times)^d \to (\C^\times)^m$ restricts to a map $\R^d_{>0} \to \R^m_{>0}$ and can be composed with the exponential map $\exp: \R^d \to \R^d_{>0}$.  We thus have exponential maps
\begin{align*}
\exp_L: &L \otimes_\Z \R \to \R^s_{>0}, \\
\exp_{L,\Z^m}: &L \otimes_\Z \R \to \R^s_{>0} \to \R^m_{>0}.
\end{align*}
Here, we may view $L \otimes_\Z \R$ as the Lie algebra of the real torus $T_{L,\R}$.

\begin{definition}
    If $P \subset L \otimes_\Z \R$ is a full-dimensional polytope, then $T_P = T_L$.  We call $X(L) \cong X(L,p)$ a \emph{toric compactification} of $T_P$.
\end{definition}
Note that if $X(L,p)$ is a toric compactification of $T_P$ then $\exp(P)$ lies in the positive real points of the dense torus orbit of $X(L,p)$.  The point $p$ is equal to $\exp(\z_0)$ in the notation of \cref{sec:expP}.


Let $H$ be a rational subspace of $L \otimes_\Z \R$.  Then $L' := H \cap L$ must be a saturated sublattice of $L$. Since $L$ itself is a saturated sublattice of $\Z^m$, this implies that $L'$ is also a saturated sublattice of $\Z^m$. Since $T_{L'} \subset T_L$, for a point on the shifted subtorus $q \in T_L \cdot p$, we have  $T_{L'} \cdot q \subset T_L \cdot p \subset T$. This implies that $T_{L'} \cdot q$ is a shifted subtorus of both $T_L$ and $T$. Therefore the projective toric variety $X(L',q)$ associated to the pair $(L', q)$ can be viewed as a subvariety of the projective toric variety $X(L,p)$ associated with $(L,p)$, or directly as a subvariety of the larger projective toric variety $\Pbb^m$.

\begin{proposition}
Under the exponential map $\exp_L$ (resp. $\exp_{L,\Z^m}$), the Zariski closure of the image of a rational linear subspace $H \subset L \otimes_\Z \R$ is a shifted subtorus of $T_L$ (resp. $T$). 
\end{proposition}

The ideal of a shifted subtorus is a binomial ideal.

\begin{defn}
        The \emph{lattice ideal} $I_M \subset \C[x_1,\ldots,x_m]$ of a saturated lattice $M \subset \Z^m$ is
$$
I_M := \langle \x^\u - \x^\v \mid \u,\v \in \Z^m \text{ with } \u - \v \in M \rangle.
$$
\end{defn}

Define a semigroup $\Gamma := \Nbb^m/\!\sim_M$ where $\sim_M$ is the equivalence relation $\u \sim_M \v \iff \u-\v \in M$.  The ideal $I_M$ is prime and the quotient $\C[x_1,\ldots,x_m]/I_M$ is the semigroup ring of $\C[\Gamma]$ \cite[Chapter 7]{Miller-Sturmfels}.

Let $\M$ denote a $m \times r$ matrix whose column vectors $\bm_1,\ldots,\bm_r$ generate $M$.  As an ideal in $\C[x_1^{\pm 1},\ldots,x_m^{\pm 1}]$, $I_M$ is generated by $\x^{\bm_i^+} - \x^{\bm_i^-}$, where $\bm_i = \bm_i^+ - \bm_i^-$ and $\bm^{\pm}_i$ extracts the positive (resp. negative) entries of $\bm_i$.

Suppose $L$ is a saturated sublattice of $\Z^m$, then $M := (\Z^m/L)^\vee \subset (\Z^m)^\vee$ is also a saturated sublattice.
We have the following dual short exact sequences of lattices.
\begin{align*}
&0 \to L \to \Z^m \to M^\vee \to 0; \\
&0 \to M \to (\Z^m)^\vee \to L^\vee \to 0.
\end{align*}
Taking $\Hom(-,\C^\times)$ of the first sequence we obtain the short exact sequence of tori:
$$
1\to (\C^\times)^d \to (\C^\times)^m \to (\C^\times)^r \to 1,
$$
where the map $(\C^\times)^m \to (\C^\times)^r$ sends 
$$(a_1,\ldots,a_m) \mapsto (\a^{\bm_1},\ldots,\a^{\bm_r}).$$
The lattices $L^\vee, (\Z^m)^\vee, M$ are the character lattices of the three tori.  The torus $T_L = V(I_M)$ is the subvariety of $T = (\C^\times)^m$ cut out by the ideal $I_M$.  

\subsection{Exponential polytopes}
\label{sec:expP}

We introduce \emph{exponential polytopes}, with the main goal of showing that they are positive geometries whose canonical forms are given by $K_P(\y)\omega_{T_P}$. 

Let $H$ be an affine subspace in  $L \otimes_{\Z} \R \simeq \R^d$, and let $\exp$ be the exponential function. The image $\exp(H)$
is called a \emph{binomial subspace} in $ (\R^\times)^d$, and we denote it by $T_H$. If $H$ is given by linear equations $A \z = \mathbf{b}$ with the matrix $A = (a_{ij})$, then $T_H$ is given by the following binomial equations in variables $\y = \exp(\z)= (y_1, y_2, \dots, y_d)$:
$$
\y^A = \exp(\b), \qquad \text{or equivalently} \qquad y_1^{a_{i1}}y_2^{a_{i2}} \cdots y_d^{a_{id}} = \exp(b_i) \; \text{ for } i=1,2,\ldots d.
$$


\begin{definition}
Let $P$ be a polytope.  The image $\exp(P)$ is called an \emph{exponential polytope}. The \emph{faces} of $\exp(P)$ are defined as the image of the faces of $P$ under $\exp$. We also refer to them as the \emph{exponential faces} in order to distinguish them from the faces of $P$.
\end{definition}

Since the exponential map is injective, every face $\exp(F)$ of $\exp(P)$ is itself an exponential polytope, and the intersection of two faces of $\exp(P)$ is again a face of $\exp(P)$.

Let $H_P\subset L \otimes_{\Z} \R$ be the affine span of $P$. The \emph{binomial subspace spanned by $\exp(P)$} is the binomial subspace spanned by $H_P$, denoted $T_P (= T_{H_P})$.  If $P$ is full-dimensional then $T_P$ is the entire algebraic torus $L\otimes_{\Z}\C^{\times} \simeq (\C^\times)^d$.  If $P$ is not full-dimensional and $\z_0 \in H_P$ is chosen as the origin for the sublattice $L' = L_{H_P,\z_0}$ of $L$ on $H_P$, then $\exp(\z_0)$ serves as the identity for the binomial subspace $T_P$.

For each exponential polytope $\exp(P)$ with $P$ full-dimensional in $L \otimes_\Z \R$, the binomial subspace $T_P$ spanned by it is equipped with the following form
$$
\omega_{T_P} := \frac{d\y}{\y} = \frac{dy_1}{y_1} \wedge \frac{dy_2}{y_2} \wedge  \cdots \wedge \frac{dy_d}{y_d}.
$$

Notice that $\omega_{T_P} = \omega_{T_L}$ is an invariant of $L$, and does not depend on the choice of the origin $\z_0$ in $L$; a translation in $\z$-coordinates is a scalar multiplication in $\y$-coordinates, which preserves the form $\omega_{T_P}$.  

Consider the \keydef $K_P(\y)$, which is a rational function in $y_1, \dots, y_d$. Together with the form $\omega_{T_P}$, we define the following rational form on $T_P$:
\begin{align}\label{eq: Theta_P}
\Theta_P:= K_P(\y) \omega_{T_P} = K_P(\y) \frac{d\y}{\y}.
\end{align}

Since neither $K_P(\y)$ nor $\omega_{T_P}$ depends on the choice of the origin 
the same holds for $\Theta_P$.

\begin{ex}\label{ex: 2-dimold}
Let $L$ be the standard lattice $\mathbb{Z}^2$ in $\R^2$. Let $P$ be the convex hull of $(-2,-1),(0,1), (1,0)$. The vectors $\u_1 = (1,-1)$, $\u_2 = (-1,-1)$ and $\u_3 = (-1,3)$ are the integral generating vectors of the rays of its normal fan $\mathcal{N}(P)$.  The polytope $P$ and the exponential polytope $\exp(P)$ are shown in Figure~\ref{fig:exp-2-dim}.  



One computes the dual lattice function $K_P(\z)$ to be equal to
\begin{eqnarray*}
 && \frac{(1-e^{-1})(1+e^{-1}+e^{-1+z_1}+e^{-1+z_2}+e^{-1+z_1-z_2}+e^{-1+z_1-2z_2}+e^{-1-z_2}+e^{-2+z_1-z_2})}{(1-e^{-1+z_1+z_2}) (1-e^{-1-z_1+z_2})(1-e^{-1+z_1-3z_2})}.
\end{eqnarray*}

The affine span is $H=\R^2$ and the binomial subspace is $T_H=(\R^{\times})^2$. 
The denominator of $K_P(\y)$ has three binomial factors: 
\[ (1-e^{-1}y_1y_2), (1-e^{-1}y_1^{-1}y_2), (1-e^{-1} y_1y_2^{-3}), \]
each corresponding to one of the three facets of $\exp(P)$ shown in the figure. 

Each facet hypersurface is a binomial subspace, and the Zariski closure of these hypersurfaces in $\Pbb^2$ has normalizations isomorphic to the projective line $\Pbb^1$.  In general, the Zariski closures of the facet hypersurfaces are interesting toric varieties.  

\begin{figure}[h!]
\centering
\begin{tikzpicture}[scale=1.0]
\draw[fill=black,opacity=0.2](0,1)--(1,0)--(-2,-1)--(0,1);
\draw[thick](0,1)--(1,0)--(-2,-1)--(0,1);
\node at (0,1) {$\bullet$};
\node at (1,0) {$\bullet$};
\node at (-2,-1) {$\bullet$};
\node[above] at (0,1) {$(0,1)$};
\node[above] at (1,0) {$(1,0)$};
\node[below] at (-2,-1) {$(-2,-1)$};
\draw[thin](0,-1)--(0,2);
\draw[thin](-2,0)--(2,0);
\end{tikzpicture}
\qquad
\begin{tikzpicture}[scale=1.0]

\draw[thin](0,0)--(3.4,0);
\draw[thin](0,0)--(0,3.2);
\node[right] at (3.4,0) {$y_1$};
\node[above] at (0,3.2) {$y_2$};



\draw[fill=black,opacity=0.15]
  plot[domain=1:2.718,samples=80] (\x,{2.718/\x})
  --
  plot[domain=1:0.368,samples=80] ({2.718*\x*\x*\x},{\x})
  --
  plot[domain=0.135:1,samples=80] (\x,{2.718*\x})
  -- cycle;

\draw[thick]
  plot[domain=1:2.718,samples=80] (\x,{2.718/\x});
\draw[thick]
  plot[domain=0.368:1,samples=80] ({2.718*\x*\x*\x},{\x});
\draw[thick]
  plot[domain=0.135:1,samples=80] (\x,{2.718*\x});

\node at (1,2.718) {$\bullet$};
\node at (2.718,1) {$\bullet$};
\node at (0.135,0.368) {$\bullet$};

\node[above] at (1,2.718) {$(1,e)$};
\node[right] at (2.718,1) {$(e,1)$};
\node[left] at (0.135,0.368) {$(e^{-2},e^{-1})$};

\node at (2.65,1.75) {$y_1y_2=e$};
\node at (-0.3,1.53) {$y_2=e y_1$};
\node at (1.55,0.45) {$y_1=e y_2^3$};
\end{tikzpicture}
\caption{A polytope $P$ in $\R^2$ and its exponential polytope $\exp(P)$}
\label{fig:exp-2-dim}
\end{figure}
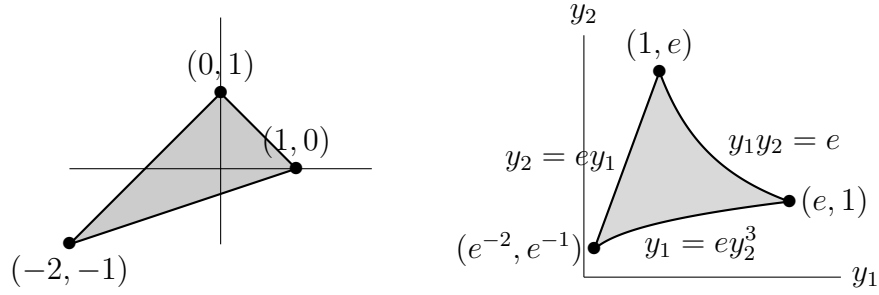
\end{ex}

\begin{ex}
Let $L$ be the standard lattice $\Z^3$ in $\R^3$. Let $P$ be the convex hull of $(0,0,0),(2,-1,0),(-1,2,0),(0,0,2)$, which is a non-unimodular simplex. See Figure~\ref{fig:exp-3-dim}. Its exponential polytope is \[\exp(P)=\{(y_1,y_2,y_3)\in \R_{>0}^3:
y_3\geq 1,\; y_1y_2^2\geq 1,\; y_1^2y_2\geq 1,\;
y_1^2y_2^2y_3\leq e^2\}.\]
Consider the facet $F$, which is the convex hull of $(2,-1,0),(-1,2,0),(0,0,2)$. It lies in the affine subspace $2z_1+2z_2+z_3=2$. Hence $\exp(F)$ is contained in
$T_F=\{y_1^2y_2^2y_3=e^2\}\subset(\C^\times)^3$. In the standard compactification $(\C^\times)^3\subset \Pbb^3$ with
$[X_0:X_1:X_2:X_3]=[1:y_1:y_2:y_3]$, the Zariski closure is
\[
\overline{T_F}=\{X_1^2X_2^2X_3=e^2X_0^5\}\subset \Pbb^3.
\]
\begin{figure}[h!]
\centering
\begin{tikzpicture}[scale=1.45,
    x={(0.80cm,-0.17cm)}, y={(0.45cm,0.31cm)}, z={(0cm,0.66cm)}]
\coordinate (O) at (0,0,0);
\coordinate (A) at (2,-1,0);
\coordinate (B) at (-1,2,0);
\coordinate (C) at (0,0,2);

\fill[black,opacity=0.16] (A)--(B)--(C)--cycle;
\draw[thick] (O)--(A)--(B)--(O)--(C)--(A);
\draw[thick] (B)--(C);

\draw[thin] (-1.35,0,0)--(2.45,0,0) node[below] {$z_1$};
\draw[thin] (0,-1.35,0)--(0,2.45,0) node[right] {$z_2$};
\draw[thin] (0,0,-1.15)--(0,0,2.45) node[above] {$z_3$};

\node[anchor=north] at (A) {\normalsize $(2,-1,0)$};
\node[anchor=east] at (B) {\normalsize $(-1,2,0)$};
\node[anchor=west] at (C) {\normalsize $(0,0,2)$};
\end{tikzpicture}
\quad
\begin{tikzpicture}[scale=1.20,
    x={(0.62cm,-0.09cm)}, y={(0.31cm,0.19cm)}, z={(0cm,0.22cm)}]
\coordinate (EO) at (1,1,1);
\coordinate (EA) at (7.389,0.368,1);
\coordinate (EB) at (0.368,7.389,1);
\coordinate (EC) at (1,1,7.389);

\draw[thin] (0,0,0)--(8.25,0,0) node[right] {$y_1$};
\draw[thin] (0,0,0)--(0,11.00,0) node[right,xshift=0pt,yshift=0pt] {$y_2$};
\draw[thin] (0,0,0)--(0,0,8.35) node[left] {$y_3$};

\draw[thick] plot[variable=\t,domain=0:1,samples=70]
  ({exp(2*\t)},{exp(-\t)},1);
\draw[thick] plot[variable=\t,domain=0:1,samples=70]
  ({exp(-\t)},{exp(2*\t)},1);
\draw[thick] plot[variable=\t,domain=0:1,samples=60]
  (1,1,{exp(2*\t)});
\foreach \v in {0.2,0.4,0.6,0.8} {
  \draw[gray,thin] plot[variable=\u,domain=0:{1-\v},samples=35]
    ({exp(2-3*\u-2*\v)},{exp(-1+3*\u+\v)},{exp(2*\v)});
}
\foreach \u in {0.2,0.4,0.6,0.8} {
  \draw[gray,thin] plot[variable=\v,domain=0:{1-\u},samples=35]
    ({exp(2-3*\u-2*\v)},{exp(-1+3*\u+\v)},{exp(2*\v)});
}
\draw[very thick,gray] plot[variable=\t,domain=0:1,samples=70]
  ({exp(2-3*\t)},{exp(-1+3*\t)},1);
\draw[very thick,gray] plot[variable=\t,domain=0:1,samples=70]
  ({exp(2-2*\t)},{exp(-1+\t)},{exp(2*\t)});
\draw[very thick,gray] plot[variable=\t,domain=0:1,samples=70]
  ({exp(-1+\t)},{exp(2-2*\t)},{exp(2*\t)});

\foreach \P in {EO,EA,EB,EC} {
  \node at (\P) {$\bullet$};
}
\node[anchor=north,xshift=0pt,yshift=0pt] at (EO)
  {\normalsize $(1,1,1)$};
\node[anchor=south,xshift=0pt,yshift=0pt] at (EA)
  {\normalsize $(e^2,e^{-1},1)$};
\node[anchor=south,xshift=-5pt,yshift=0pt] at (EB)
  {\normalsize $(e^{-1},e^2,1)$};
\node[anchor=south,yshift=4pt] at (EC)
  {\normalsize $(1,1,e^2)$};
\node at (3.50,2.50,3.50)
  {\normalsize $y_1^2y_2^2y_3=e^2$};
\end{tikzpicture}
\caption{A polytope $P$ in $\R^3$ and its exponential polytope $\exp(P)$}
\label{fig:exp-3-dim}
\end{figure}
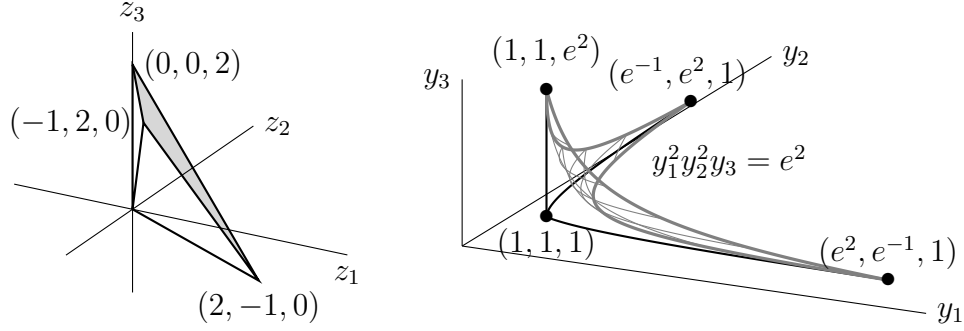
\end{ex}

\begin{prop}\label{prop:posconv}
The function $K_P(\y)$ takes positive values in the interior of $\exp(P)$.
\end{prop}
\begin{proof}
Let $P$ be full-dimensional in $L \otimes_\Z \R \cong \R^d$.
    When $\z \in \Int(P)$, the image of the support function $h_{P-\z}$ is positive everywhere on $\R^d \setminus {\bf 0}$, and thus $\exp(-h_{P-\z}(\v)) < 1$ for any $\v \neq {\bf 0}$.

By \cref{lem: parallelepiped formula interior} and \cref{thm: sum-of-cones}, $K_P(\z)$ can be written as a sum of terms $K^\circ_{P,C}(\y)$ indexed by cones $C$ in a triangulation of the normal fan of $P$. Each summand $K^\circ_{P,C}(\y)$ is a rational function whose numerator is a positive sum of powers of the positive number $y_0$, and whose denominator is a product of factors of the form $(1- \exp(-h_{P-\z}(\v)))$. Therefore all of the factors are positive when $\z \in \Int(P)$, or equivalently, when $\y \in \Int(\exp(P))$.
\end{proof}

In the nomenclature of \cite{ABL}, \cref{prop:posconv} says that $\exp(P)$ is a \emph{positively convex} geometry.

\subsection{Toric polytopes}\label{sub:toric-polytopes}

We now provide a generalization of polytopes, in the framework of positive geometry, to include objects such as exponential polytopes.

\begin{definition}\label{def:toricpolytope}
    A positive geometry $(X,E)$ with canonical form $\Omega_E$ is a \emph{toric polytope} if it satisfies the following.
\begin{enumerate}
    \item $X$ is a projective normal toric variety, 
    \item either $E$ is a point, or each boundary component of $(X,E)$ is a toric polytope, 
    \item $E$ is homeomorphic to a closed ball,
    \item $E$ is \emph{positive convex}, that is, the canonical form $\Omega_E$ takes constant sign in $\Int(E)$.
\end{enumerate}

\end{definition}

\begin{remark}
    It is also natural to make \cref{def:toricpolytope} stronger by insisting that the toric variety $X_F$ associated to a facet (``boundary component'') $F$ of a toric polytope $(X,E)$  is not just an abstract toric variety, but is a toric subvariety of $X$ in the sense of \cref{sec:toricsub}.  We call such positive geometries \emph{strong toric polytopes}.
\end{remark}


It is immediate to check that a (projective) polytope $P$ with its canonical form as in \cite{GLX} is a toric polytope as in \cref{def:toricpolytope}. Our main result in this subsection is that exponential polytopes also fit in this definition.

\begin{theorem}\label{thm:expP2}
    An exponential polytope $\exp(P)$ is a toric polytope inside the projective toric variety $X(L)$, with its canonical form equal to $\Theta_P$ as defined in \eqref{eq: Theta_P}.
\end{theorem}

Let $Q_L$ be the polytope associated to $L$ as defined previously. (Note that $Q_L$ has no immediate relation to $P$.)  For a facet $F$ of the polytope $Q_L$, let $D(F) \subset X(L,p)$ denote the corresponding torus orbit closure, which is a divisor. Let $\Theta = f(\y)/g(\y) \omega_{T_L}$ be a rational form on $T_L$, viewed as a rational form on $X(L,p) \cong X(L)$.

\begin{lemma}\label{lem:degw}
Let $F$ be a facet of $Q_L$.
    The order of the pole of $\Theta = f(\y)/g(\y) \omega_{T_L}$ on $D(F)$ is equal to $\deg_{\w}( f(\y)/g(\y)) + 1$ where $\w$ is the primitive integer outward-pointing normal vector to the facet $F$.
\end{lemma}
\begin{proof}
    This is a local calculation, and can be reduced to a calculation on the affine toric variety $\C \times (\C^\times)^{d-1}$ which has a fan consisting of a single ray.  By a change of basis, we may assume that the ray is the $y_1$-axis, and we are reduced to the following statement.  The order of pole of $\frac{f(y)}{g(y)}\frac{dy}{y}$ at $y=\infty$ in $\Pbb^1$ is $1 + \deg(f(y)) - \deg(g(y)) $, which is true.
\end{proof}

\begin{proof}[Proof of \cref{thm:expP2}]
    
We first verify that the pair $(X(L), \exp(P))$, equipped with $\Theta_P$, is a positive geometry.  We need to show that $\Theta_P$ has no poles other than simple ones on the boundary components corresponding to facets of $P$. First note that the form $\omega_{T_P}$ has a simple pole along each boundary divisor of $X(L)$, thus by \cref{lem:degw} and \cref{prop:deg}, the rational form $\Theta_P$ has no poles along any of the boundary divisors of the toric variety $X(L)$ at infinity. Next we check inside the torus $T_P$. By \cref{lem: parallelepiped formula interior} and \cref{thm: sum-of-cones}, $K_P(\y)$ can be written as a sum of local terms $K_{P,F}(\y)$ for the facets $F$ of $P$. In each summand, the denominator factors into terms of the form $1- \frac{y_0^{\u_i\cdot \p }}{\y^{\u_i}}$. After cancellation, the remaining factors correspond to simple poles at the subtori $T_F$ in $T_P$. Since these are the only factors that show up, $\Theta_P$ does not have other poles anywhere else in $T_P$. Therefore $\Theta_P$ has simple poles at and only at the boundary components $T_F$ corresponding to facets $F$ of $P$. In Theorem \ref{thm:ResF}, we have shown that $\Res_{T_F}\Theta_P = \Theta _F$ for each facet $F$ of $P$. Finally let's check uniqueness of the rational form. Since the toric variety $X(L)$ is rational, it has no holomorphic top-forms, so $\Theta_P$ is the unique rational top-form with the correct residues. This completes the proof that this is a positive geometry.

 We now turn to the four additional conditions in Definition \ref{def:toricpolytope}. Three of them are immediate: the exponential polytope $\exp (P)$ is homeomorphic to a closed ball; the projective toric variety $X(L)$ is normal since $L$ is a saturated sublattice; and the positive convexity of $\exp(P)$ is proved in \cref{prop:posconv}.

It remains to check each boundary component is itself a toric polytope. Each facet of $\exp(P)$ is $\exp(F)$ for some facet $F$ of $P$. The corresponding boundary component pair is $(X(L_F), \exp(F))$, where $L_F = L \cap \text{span}(F)$. Since $L$ is a saturated sublattice of $\mathbb{Z}^m$, so is $L_F$, hence $X(L_F)$ is a normal projective toric variety. Furthermore, $F$ is rational and full-dimensional in $L_F$. Therefore by induction, we have that every such boundary component is a toric polytope. Therefore by Definition \ref{def:toricpolytope}, $(X(L), \exp(P))$ itself is a toric polytope.
\end{proof}

\begin{cor}
Suppose that $Q_1,\ldots,Q_r$ is a subdivision of $P$.  Then $\Theta_P = \sum_i \Theta_{Q_i}$.    
\end{cor}
\begin{proof}
    The exponential polytopes $\exp(Q_1),\ldots,\exp(Q_r)$ give a decomposition of $\exp(P)$.  The identity then follows from general facts about decompositions of positive geometries \cite{ABL,LamOPAC}.  Alternatively, the result follows immediately from \cref{thm:valuation}.
\end{proof}

\begin{remark}
    Let $S \subset \R^d$ be a nonempty union of finitely many $d$-dimensional polytopes $Q_i$.  Then $S$ is a toric polytope if and only if $S$ is a polytope.  To see this, consider the hyperplane arrangement consisting of all facets of all the polytopes $Q_i$.  Then $S$ is also the union of some number of chambers $R_i$ of the arrangement.  Let $H$ be a hyperplane that is a facet of some region $R_i$ and such that the region bordering $R_i$ on the other side of $H$ does not belong to $S$.  The conditions (3) and (4) of \cref{def:toricpolytope} imply that the whole of $S$ must lie on the same side of $H$.  As we vary over all such $H$ we obtain the facets of $S$ as a polytope.
\end{remark}

\bibliographystyle{alpha}
\bibliography{ref}
\end{document}